\documentclass[11pt]{article}
\usepackage{latexsym,amssymb,amsmath,a4wide,theorem}

\usepackage{esint}

\numberwithin{equation}{section}
\numberwithin{figure}{section}
\newtheorem{theorem}{Theorem}[section]
\newenvironment{proof}{{ \it Proof:\quad}}{\hfill $\blacksquare$\par}
\newtheorem{lemma}{Lemma}[section]
\newtheorem{proposition}{Proposition}[section]
\newtheorem{corollary}{Corollary}[section]
\newtheorem{definition}{Definition}[section]

\newtheorem{remark}{Remark}[section]

\newcommand{\f}{\frac}
\newcommand{\lf}{\left}
\newcommand{\rg}{\right}
\newcommand\tbbint{{-\mkern -16mu\int}}

\newcommand\dbbint{{-\mkern -19mu\int}}

\newcommand\bbint{
{\mathchoice{\dbbint}{\tbbint}{\tbbint}{\tbbint}}
}

\title{Asymptotic Stability of Phase Separation States for Compressible Immiscible Two-Phase Flow with Periodic Boundary Condition in 3D}

\author{Yazhou Chen\thanks{Department of Mathematics, College of Mathematics and Physics, Beijing University of Chemical Technology, Beijing 100029,
		P R  China (chenyz@mail.buct.edu.cn).}
\and Hakho Hong\thanks{Institute of Mathematics, State Academy of Sciences, Pyongyang,
		D P R Korea (hhhong@star-co.net.kp and hhong@amss.ac.cn).}
\and Xiaoding Shi\thanks{Corresponding author. Department of Mathematics,  College of Mathematics and Physics, Beijing University of Chemical Technology, Beijing 100029,
		P R  China (shixd@mail.buct.edu.cn) }
}

\date{}

\begin{document}
	\maketitle
	\begin{abstract}
		This paper is concerned with a diffuse interface model called as Navier-Stokes/Cahn-Hilliard system. This model is usually used to describe the motion of immiscible two-phase flow with diffusion interface. For the periodic boundary value problem of this system in torus $\mathbb{T}^3$, we prove that there exists a global unique strong solution  near the phase separation state, which means no vacuum, shock wave, mass  concentration, interface collision and rupture will be developed in finite time. Furthermore, we established the large time behavior of these global strong solution of this system. In particular, we find that the phase field decays algebraically to the phase separation state.

		\vspace{.20cm}\noindent\textbf{MSC 2020:}
		35B40, 35B65, 35L65, 76N05, 76N10, 76T10.
		
		\vspace{.20cm} \noindent\textbf{Keywords:} Navier-Stokes/Cahn-Hilliard
		system, existence, uniqueness, large-time behavior.
	\end{abstract}
	
\section{Introduction}
\setcounter{equation}{0}

 \ \ \ \  Two-phase flows or multi-phase flows are important in many industrial applications, for
instance, in aerospace, chemical engineering, micro-technology and so on. They have at-
tracted studies from many engineers, geophysicists and astrophysicists. We focus on diffuse
interface model which describes the motion of a mixture of two compressible viscous flu-
ids with different densities. Macroscopically immiscible two-phase fluids are assumed to
be separated by a sharp interface. However, in order to describe topological transitions,
such as droplet formation, coalescence of several droplet or droplet breakup, we need to
take into account a partial mixing on a small length scale in the model. As a result, the
sharp interface of the two fluids is replaced by a narrow transition layer, and an order pa-
rameter related to the concentration difference of both fluids is introduced.

In this paper, we consider the following barotropic  compressible Navier-Stokes/Cahn-Hilliard system (see Heida-M\'alek-Rajagopal \cite{HMR 12} and Abels-Feireisl \cite{AF 08}) describing immiscible two-phase flow with diffuse interface:
\begin{equation}\label{11h}\begin{cases}
\displaystyle\rho_{t}+{\rm div}(\rho \mathbf{u})= 0, \\
\displaystyle\lf(\rho \mathbf{u}\rg)_t+{\rm div}(\rho \mathbf{u}\otimes \mathbf{u})+\nabla p(\rho)
={\rm div}\mathbf{S}-{\rm div}\lf(\nabla\phi\otimes \nabla\phi-\f{|\nabla\phi|^2}{2}\mathbb{I}\rg),\\
\displaystyle(\rho\phi)_{t}+{\rm div}(\rho\phi \mathbf{u})=\Delta\mu, \\
\displaystyle  \mu=-\f{\epsilon}{\rho}\Delta\phi+\frac{1}{\epsilon}\partial_\phi f(\phi),
\end{cases}
\end{equation}
where $\mathbf{x}=(x_1,x_2,x_3)\in\mathbb{T}^3$, $\mathbb{T}^3$ is the torus in $\mathbb{R}^3$, $t>0$ the time variable.
The unknown functions $\rho, \mathbf{u}$  and $\phi$ represent  the density, velocity and phase field respectively. Notice that, here $\{ \mathbf{x}:\phi(\mathbf{x},t) = 1\}$ is occupied by fluid 1 and $\{ \mathbf{x}:\phi(\mathbf{x},t) = -1\}$ by fluid 2 in the area of immiscible two-phase flow. The  phase function $\phi$ is introduced to distinguish between different fluids.
In terms of simple, taking any volume element $V$ in the flow, $M_i$ is assumed to be the mass of the components in the representative material volume $V$, $\phi_i=\frac{\rho_i}{\rho}$ the mass concentration, $\rho_i=\frac{M_i}{V}$  the apparent mass density of the fluid $i~(i=1,2)$. $\rho=\rho_1+\rho_2$ the total density, and the phase function $\phi=\phi_1-\phi_2$.  Obviously, We can determine the location of the spread interface by the phase function $\phi$.

The constant $\epsilon>0$ is the thickness of the interface between the phases,
and the deformation tensor $\mathbf{S}$ is given by
\begin{equation}\label{S}
 \mathbf{S}=2\nu(\phi) \mathbf{D}(\mathbf{u})+\lambda(\phi) {\rm div}\mathbf{u}\mathbb{I},
\end{equation}
where $\mathbf{D}(\mathbf{u})=\f{1}{2}\lf(\nabla \mathbf{u}+\nabla^{\top} \mathbf{u}\rg)$,  $\top$ represents the transpose of a matrix, $\mathbb{I}$ is the identity matrix, $\nu(\phi)$ and  $\lambda(\phi)$ are two viscosity coefficients.
In this paper, we suppose that $\nu(\cdot), \lambda(\cdot)\in C^{3}(\mathbb{R}) $  and there exists positive constant $\nu_0$ such that
\begin{equation}\label{h22}
\nu(s)\geq \nu_0>0,\quad  \lambda(s)\geq 0.
\end{equation}
Also, $p=p(\rho)$ is the pressure,  satisfies
\begin{equation}\label{h23}
p'(\rho)>0 \quad \text{for all}\,\, \rho>0.
\end{equation}
Here $f$ is the potential function.  It is reasonable to take a potential function  $f$ as follows (see  Heida-M\'alek-Rajagopal \cite{HMR 12})
\begin{equation}\label{hhh}f(\phi)=\f{\phi^4}{4}-\f{\phi^2}{2}.\end{equation}

\begin{remark}
In the theory of the Cahn-Hilliard equation, double-well structural potential is often considered.	A typical example
of such potential is of logarithmic type
\begin{equation*}\label{h12}
f(\phi)=\f{\alpha_1}{2}\lf((1-\phi)\ln(\f{1-\phi}{2})+(1+\phi)\ln(\f{1+\phi}{2})\rg)-\f{\alpha_2}{2}\phi^2,\end{equation*} where $\alpha_1$ and $\alpha_2$ are positive constants,
which is suggested by Cahn-Hilliard \cite{CH 58}.
However, this potential is usually replaced by a polynomial approximation of the type $\beta_1\phi^4-\beta_2\phi^2$,  where $\beta_1$ and $\beta_2$ are positive constants (for example, see \cite{Y 94}, \cite{DLL 13}).
Therefore,  without loss of generality, the approximate expression \eqref{hhh} is usually adopted.
\end{remark}

The main purpose of this paper is to give a rigorous analysis of the asymptotic stability of the phase separation state of immiscible two-phase flow. We consider the initial boundary problem of the system  \eqref{11h}  in torus $\mathbb{T}^3=(0,1)^3$ with the periodic boundary condition
\begin{equation}\label{g15}
(\rho, \mathbf{u}, \phi)(\mathbf{x}+\mathbf{L},t)=(\rho, \mathbf{u}, \phi)(\mathbf{x},t), \,\, t\geq0,
\,\,\, \mathbf{L}=(1,1, 1),
\end{equation}
and the initial  condition
\begin{equation}\label{g14}
(\rho, \mathbf{u}, \phi)(\mathbf{\mathbf{x}},0)=(\rho_0, \mathbf{u}_0, \phi_0)(\mathbf{x}).\quad \mathbf{x}\in \mathbb{T}^3.\end{equation}

\vspace{0.2cm} For the last few decades, the mathematical study of diffusive interface models for compressible immiscible two-phase flow has been attracted by many researchers.
For the system \eqref{11h} which can be regarded as a simple variant of the model derived originally by Heida-M$\mathrm{\acute{a}}$lek-Rajagopal \cite{HMR 12}, Abels-Feireisl \cite{AF 08} proved existence of global
weak solutions to the initial boundary value problem in bounded domain $\Omega \subset\mathbb{R}^3$. For the more general  system,
 Kotschote-Zacher \cite{KZ 15} established a local existence and uniqueness
of strong solutions for the initial boundary value problem  in bounded domain $\Omega \subset\mathbb{R}^n(n=2,3)$.
Recently, in the case of one-dimensional initial boundary value problem, Chen-He-Mei-Shi \cite{CHMS 18} studied the global existence and the large time behavior of the strong solutions, even for the large initial disturbance of the density and the large velocity data.
For the full compressible Navier-Stokes/Cahn-Hilliard system with heat-conductivity, a case more closer to the physical reality, there exist some recent progresses to derive a thermodynamically consistent model (see \cite{HMR 12, K 16, FK 17, FK 18} and so on).

It should be noted that Navier-Stokes/Allen-Cahn system is another commonly used immiscible two-phase flow model for diffusion interfaces. The essential difference between the two models is that, the phase field equation in the Navier-Stokes/Cahn-Hilliard system is a fourth order equation, which is conserved with respect to the $\rho\phi$. But the phase field equation in the Navier-Stokes/Allen-Cahn system  is a second order equation, and it is not conserved with respect to the $\rho\phi$. Therefore, the Navier-Stokes/Cahn-Hillard system is more suitable to describe the phase separation phenomenon, while the Navier-Stokes/Allen-Cahn is more suitable to describe the phase transition phenomenon. There are also many recent studies for the compressible Navier-Stokes/Allen-Cahn system recently, see  \cite{FPRS 10, K 12, DLL 13, CG 17,  LY 18, HMSW 18, CWZ 19, DLT 19, YZ 19}  and references therein.



\vspace{0.2cm}\textbf{Notation:} In this paper,
$L^p(\mathbb{T}^3)$ and $W^{k,p}(\mathbb{T}^3)$ denote the usual Lebesgue and Sobolev spaces on $\mathbb{T}^3$, with norms  $\|\cdot\|_{L^p}$ and $\|\cdot\|_{W^{k,p}}$, respectively. When $p=2$, we denote $W^{k,p}(\mathbb{T}^3)$ by $H^k(\mathbb{T}^3)$ with  the norm  $\|\cdot\|_{H^k}$ and $\|\cdot\|_{H^0}=\|\cdot\|$ will be used to denote the usual
$L^2-$norm.  The notation $\|(A_1,A_2, \cdots, A_l)\|_{H^k}$
means the summation of $\|A_i\|_{H^k}$ from $i=1$
to $i=l$.
For an integer $m$, the symbol $\nabla^m$
denotes the summation of all terms $D^{\alpha}$ with the multi-index $\alpha$
satisfying $|\alpha|=m$. We use $C, c$ to denote the constants which are
independent of $\mathbf{x},t$
and may change from line to line.
For $3\times 3$-matrices $F, H$, denote $
F:H=\sum_{i,j=1}^{3}F_{ij}H_{ij}$, $|F|\equiv(F:F)^{1/2}$.  For vectors $a$ and $b$, we denote their tensor product by $a\otimes b:=(a_ib_j)_{3\times 3}$. 
The integral mean is denoted by the following symbol
\begin{equation}\label{average}
 \bbint_{\mathbb{T}^3} \cdot d\mathbf{x}\overset{\mathrm{def}}=\frac{1}{|\mathbb{T}^3|}\int_{\mathbb{T}^3} \cdot d\mathbf{x}.
\end{equation}
 We will employ the notation $a\lesssim b$ to mean that $a\leq  Cb$ for a universal constant $C>0$ that only depends on the parameters coming from the problem.
Next, in order to establish the negative Sobolev estimates, we should review the following useful
results. But let us first introduce the following necessary definition.

\begin{definition}\label{def21}
	For $s\in\mathbb{R}$, $\dot{H}_{\mathrm{per}}^s(\mathbb{T}^3)$ is defined as the homogeneous Sobolev space of $f,$ with the periodic condition and following norm:
	$$
	\|f\|_{\dot{H}^{s}}\overset{\mathrm{def}}=\|\Lambda^s f\|,
	$$where $\Lambda^s$ is defined by
	$$(\Lambda^s f)(\mathbf{x})\overset{\mathrm{def}}=\sum_{\mathbf{k}\in \mathbb{Z}^3}(2\pi\mathbf{k})^s \hat{f}(\mathbf{k})e^{2\pi i \mathbf{x}\cdot \mathbf{k}}, \quad \mathbf{k}=(k_1,\cdots, k_3),$$ where $\hat{f}(\mathbf{k})$ is the Fourier transform of $f$ defined by
	$$\hat{f}(\mathbf{k})=\int_{\mathbb{T}^3}f(\mathbf{x})e^{-2\pi i \mathbf{x}\cdot \mathbf{k}} d\mathbf{x}.$$
\end{definition}
Before giving the main results, for convenience, we present several important Lemmas commonly used in this paper.
We set
\begin{equation}\label{periodic function space}
  W^{k,q}_{\mathrm{per}}(\mathbb{T}^3)\overset{\mathrm{def}}=\Big\{w\in W^{k,q}(\mathbb{T}^3)\Big| w(\mathbf{x}+\mathbf{L})=w(\mathbf{x})\Big\}.
\end{equation}
Then, we  need the following result which is a direct consequence of Gagliardo-Nirenberg inequality  (\cite{N 59}):
\begin{lemma} \label{lem21}
	Let $l,s$ and $k$ be any real numbers satisfying $0\leq
	l,s<k$, and let  $p, r, q \in [1,\infty]$ and $\f{l}{k}\leq
	\theta\leq 1$ such that
	$$\frac{l}{3}-\frac{1}{p}=\lf(\frac{s}{3}-\frac{1}{r}\rg)(1-\theta)+\lf(\frac{k}{3}-\frac{1}{q}\rg)\theta.
	$$ Then,  we have
	\begin{equation}
	\label{hg20}
	\|\nabla^l w\|_{L^p}\lesssim \|\nabla^s w\|_{L^r}^{1-\theta}\|\nabla^k	w\|_{L^q}^{\theta},
	\end{equation} for any $\displaystyle w\in W^{k,q}_{\mathrm{per}}(\mathbb{T}^3)$.
\end{lemma}

Next, we need the following result:
\begin{lemma} \label{lem22-1}
	  Let $f(\sigma)$ and $f(\sigma, w)$ be  smooth functions of $\sigma$ and $(\sigma, w)$, respectively, with bounded
	derivatives of any order, and $\|\sigma\|_{L^\infty(\mathbb{T}^3)}\lesssim 1$. Then for any integer $m\geq 1$, we have
	\begin{equation}	\label{hg21}
\begin{aligned}
	\|\nabla^m f(\sigma)\|_{L^p}\leq C \|\nabla^m \sigma\|_{L^p},\quad
\|\nabla^m f(\sigma,w)\|_{L^p}\leq C  \|\nabla^m (\sigma,w)\|_{L^p},\quad \forall 1\leq p\leq \infty,\end{aligned}
	\end{equation}
 where $C$	may depend $f$ and $m$.	
\end{lemma}
\begin{proof} The \eqref{hg21} can be obtained directly making use of the Gagliardo-Nirenberg inequality (Lemma 1.1), we therefore omitted here for the sake of brevity.
\end{proof}

\vspace{0.3cm}Next, we recall the following Moser-type calculus inequalities (\cite{KM 81}):
\begin{lemma} \label{lem22}   Let $\Omega$ be a domain of $\mathbb{R}^3$. Then, 	
	for $f, g\in H^s(\Omega)\cap L^\infty(\Omega)$ and $|\alpha|\leq s,\,s>\f{3}{2}$, it holds that \begin{equation}
	\label{hg22}
	\|D^\alpha(fg)\|\lesssim \|f\|_{L^\infty}\|\nabla^s g\|+\|g\|_{L^\infty}\|\nabla^s f\|.
	\end{equation}
		
\end{lemma}

\vspace{0.3cm}By the Parseval theorem and H\"older's inequality, it is easy to check  the following result (see \cite{GW 11}).
\begin{lemma} \label{lem24}
	Let $s\geq 0$ and $l\geq 0$. Then, we have
	\begin{equation}
	\label{hg24}\|\nabla^lf\|\lesssim\|\nabla^{l+1}f\|^{1-\theta}\|f\|^{\theta}_{\dot{H}^{-s}},\quad \mathrm{with}\,\,\,\theta=\f{1}{l+s+1}.\end{equation}
\end{lemma}

\vspace{0.3cm}If $s\in (0, 3)$, $\Lambda^{-s}g$ is the Riesz potential. Then, we have the following $L^p$ type inequality  by the discrete Hardy-Littlewood-Sobolev inequality (see \cite{CL2014},\cite{HLY2015},\cite{S 70}):

\begin{lemma} \label{lem25}
	Let $0<s<3, 1<p<q<\infty$ and $\f{1}{q}+\f{s}{3}=\f{1}{p}$. Then, we have
	\begin{equation}
	\label{hg25}\|\Lambda^{-s}f\|_{L^q}\lesssim\|f\|_{L^p}.\end{equation}
\end{lemma}

\vspace{0.3cm}Without loss of generality, we assume the thickness of the interface $\epsilon=1$ here, moreover, noticing that ${\rm div}\lf(\nabla\phi\otimes \nabla\phi\rg)=\nabla\lf(\f{|\nabla\phi|^2}{2}\rg)+\nabla\phi\Delta\phi$,
combining with  \eqref{hhh}-\eqref{g14},  \eqref{11h} is simplified to
\begin{equation}\label{h21}\begin{cases}\begin{aligned}
&\displaystyle\rho_{t}+{\rm div}(\rho \mathbf{u})= 0, \\
&\displaystyle
\lf(\rho \mathbf{u}\rg)_t+{\rm div}(\rho \mathbf{u}\otimes \mathbf{u})+\nabla p(\rho)+\nabla\phi\Delta\phi
=2{\rm div}\lf[\nu(\phi) \mathbf{D}(\mathbf{u})\rg]+\nabla\lf[\lambda(\phi) {\rm div}\mathbf{u}\rg],
	\\
&\displaystyle\rho\phi_{t}+\rho\mathbf{u}\cdot \nabla \phi =\Delta\mu, \\
&\displaystyle \rho\mu=-\Delta\phi+\rho\lf(\phi^3-\phi\rg), 
\\
&\displaystyle(\rho, \mathbf{u}, \phi)(\mathbf{x}+\mathbf{L},t)=(\rho, \mathbf{u}, \phi)(\mathbf{x},t),
\\
&\displaystyle(\rho, \mathbf{u}, \phi)(\mathbf{x},0)=(\rho_0, \mathbf{u}_0, \phi_0)(\mathbf{x}).
\end{aligned}\end{cases}
\end{equation}
Now, the main results is as following below:
\begin{theorem}  \label{theo 2.1}
	Assume that \eqref{h22}-\eqref{hhh} and 
\begin{equation}\label{A1}
\rho_0, \mathbf{u}_0, \phi_0\in
	H_{\mathrm{per}}^{3}(\mathbb{T}^3),\quad \inf_{\mathbf{x}\in \mathbb{T}^3 }\rho_0(\mathbf{x})>0.
\end{equation}
Then, there exists a  positive constant  $\delta>0$  such that if
	\begin{equation}\label{g17}
	\|\rho_0-\bar{\rho}\|_{H^{3}}+\|\mathbf{u}_0\|_{H^{3}}+\|\nabla\phi_0\|_{H^{2}}+\|\phi_0^2-1\|_{L^{2}}\leq	\delta,
	\end{equation}	where $\displaystyle\bar{\rho}=\bbint_{\mathbb{T}^3}\rho_0(\mathbf{x})d\mathbf{x},$ then
	the system \eqref{h21} admits a unique solution $(\rho,
	\mathbf{u}, \phi)$ on $[0,\infty)$ satisfying
	$$
	\begin{aligned}
	&\rho, \mathbf{u}, \phi \in C([0,\infty);H_{\mathrm{per}}^{3}(\mathbb{T}^3)),\,\,\,\rho \in L^2(0,\infty;H_{\mathrm{per}}^{3}(\mathbb{T}^3)),\\
	&\nabla \mathbf{u}\in L^2(0,\infty;H_{\mathrm{per}}^{3}(\mathbb{T}^3)),\,\,\,\nabla\phi \in L^2(0,\infty;H_{\mathrm{per}}^{4}(\mathbb{T}^3)).
	\end{aligned}$$
	and
	\begin{equation}\label{h18}
	\begin{aligned}
	&\displaystyle\|(\rho-\bar{\rho}, \mathbf{u})(t)\|^2_{H^{3}}+\|\nabla\phi(t)\|^2_{H^{2}}+\|\phi^2(t)-1\|_{L^{2}}^2
	\\
	&\displaystyle+\int_0^t \| \rho(\tau)-\bar{\rho}\|^2_{H^{3}}d\tau+\int_0^t \|\nabla \mathbf{u}(\tau)\|^2_{H^{3}}d\tau+\int_0^t \|\nabla\phi(\tau)\|^2_{H^{4}}d\tau\\
	&\displaystyle\leq  C
	\lf(\|(\rho_0-\bar{\rho}, \mathbf{u}_0)\|^2_{H^{3}}+\|\nabla\phi_0\|^2_{H^{2}}+\|\phi^2_0-1\|_{L^{2}}^2\rg),\ \ \forall t\geq 0,\end{aligned}\end{equation} where $C$ is the positive constant independent of
	$\mathbf{x}, t$ and  $ \delta$.	

Moreover, if  $(\rho_0-\bar{\rho}, \mathbf{u}_0, \nabla\phi_0, \phi_0^2-1)\in \dot{H}^{-s}$ for some $s\in [0,\f{3}{2})$,  we have the following algebraic decay estimates:
		\begin{equation}
	\label{h124}
	\|(\rho-\bar{\rho}, \mathbf{u}, \nabla\phi, \phi^2-1)(t)\|_{\dot{H}^{-s}}\leq C_0,
	\end{equation}
	\begin{equation}
	\label{hg124}
	\|\mathbf{u}(t)\|_{L^{2}}^2+\| \phi^2(t)-1\|_{L^2}\leq C_0(1+t)^{-s},
	\end{equation}
	and \begin{equation}
	\label{h125}
	\|\rho(t)-\bar{\rho} \|_{H^{3}}^2+\|\nabla\mathbf{u}(t)\|_{H^{2}}^2+\|\nabla \phi(t)\|_{H^{2}}^2\leq C_0(1+t)^{-(2+s)},
	\end{equation}  where $\dot{H}^{-s}$ denotes the homogeneous negative Sobolev space with periodic condition (see Definition \ref{def21}).
\end{theorem}

By Lemma \ref{lem25}, we obtain that for $p\in(1,2]$, $L^p(\mathbb{T}^3)\subset \dot{H}^{-s}(\mathbb{T}^3)$ with $s=3\lf(\f{1}{p}-\f{1}{2}\rg)\in [0, \f{3}{2})$. Then by Theorem \ref{theo 2.1}, we have the following corollary of the usual $L^p-L^2$ type of decay results:
\begin{corollary}  \label{coro12}
	Under the assumptions of Theorem \ref{theo 2.1}, if
	$(\rho_0-\bar{\rho}, \mathbf{u}_0, \nabla\phi_0, \phi_0^2-1)\in L_{\mathrm{per}}^p(\mathbb{T}^3)$ for some $p\in(1,2]$, then the following decay results hold:
$$
	\label{h126}
	\|\mathbf{u}(t)\|_{L^{2}}^2+\| \phi^2(t)-1\|_{L^{2}}^2\lesssim  (1+t)^{-3\lf(\f{1}{p}-\f{1}{2}\rg)},
$$
	and $$
	\|\rho(t)-\bar{\rho} \|_{H^{3}}^2+\|\nabla\mathbf{u}(t)\|_{H^{2}}^2+\|\nabla \phi(t)\|_{H^{2}}^2\lesssim (1+t)^{-3+1-3\lf(\f{1}{p}-\f{1}{2}\rg)}.
$$
\end{corollary}

\begin{remark}
 The constraint
	$s<\f{3}{2} $ in Theorem \ref{theo 2.1}  comes from applying Lemma \ref{lem25}  to estimate the nonlinear terms when doing
	the negative Sobolev estimates via $\Lambda^{-s}$. This
	in turn restricts $p>1$ in Corollary \ref{coro12} by our method. Note that the nonlinear estimates would not work for $s\geq\f{3}{2}$.
\end{remark}
\begin{remark}
A particular feature of the system \eqref{h21} is that the viscosity coefficients $\nu, \lambda$ depend on the mass concentration
difference $\phi$, which makes our arguments to be difficult. In general, the viscosity coefficients $\nu, \lambda$ may depend not only on the mass concentration
difference $\phi$, but also the density $\rho$. Here, we'd like to emphasis that the result of Theorem \ref{theo 2.1} holds  by the same lines in this case. Moreover, when the dimension of the space is 2 dimensions, the conclusion of the theorem 1.1 still holds.
\end{remark}

\begin{remark}
The conclusion of the Theorem 1.1 shows that, in the small perturbation condition near the phase separation state, discontinuous phenomena such as shock wave, vacuum, interface rupture for compressible immiscible two-phase flow will not occur in finite time. Moreover,  velocity, density and phase field of immiscible two-phase flow decay algebraically with time. In particular, the interface of two-phase flow algebra decays to sharp interface limit.
\end{remark}



\begin{corollary}  \label{corollary 3}
	Assume that \eqref{h22}-\eqref{hhh}, $d$ represents the dimension of space,  and for an integer $N\geq3$,
$$
	\rho_0, \mathbf{u}_0, \phi_0\in
	H^{N}(\mathbb{T}^d),\quad \inf_{x\in \mathbb{T}^d }\rho_0(x)>0.
$$
	So, there exists a  positive constant  $\delta_0>0$  such that if
	\begin{equation}\label{g17}
	\|\rho_0-\bar{\rho}\|_{H^{N}}+\|\mathbf{u}_0\|_{H^{N}}+\|\nabla\phi_0\|_{H^{N-1}}+\|\phi_0^2-1\|\leq
	\delta_0,
	\end{equation} then
	the system \eqref{h21} admits a unique solution $(\rho,
	\mathbf{u}, \phi)$ on $[0,\infty)$ satisfying
	$$
	\begin{aligned}
	&\rho, \mathbf{u}, \phi \in C([0,\infty);H^{N}(\mathbb{T}^d)),\,\,\,\rho \in L^2(0,\infty;H^{N}(\mathbb{T}^d)),\\
	&\nabla \mathbf{u}\in L^2(0,\infty;H^{N}(\mathbb{T}^d)),\,\,\,\nabla\phi \in L^2(0,\infty;H^{N+1}(\mathbb{T}^d)).
	\end{aligned}$$
	and
	\begin{equation}\label{h18*}
	\begin{aligned}
	&\|(\rho-\bar{\rho}, \mathbf{u})(t)\|^2_{H^{N}}+\|\nabla\phi(t)\|^2_{H^{N-1}}+\|\phi^2(t)-1\|^2
	\\
	&+\int_0^t \| \rho(\tau)-\bar{\rho}\|^2_{H^{N}}d\tau+\int_0^t \|\nabla \mathbf{u}(\tau)\|^2_{H^{N}}d\tau+\int_0^t \|\nabla\phi(\tau)\|^2_{H^{N+1}}d\tau\\
	&\leq  C
	\lf(\|(\rho_0-\bar{\rho}, \mathbf{u}_0)\|^2_{H^{N}}+\|\nabla\phi_0\|^2_{H^{N-1}}+\|\phi^2_0-1\|^2\rg),\end{aligned}\end{equation} for any $t\geq 0,$ where $C$ is the positive constant independent of
	$x, t$ and  $ \eta_0$.
		Moreover, we have the following algebraic decay estimates:
if  $(\rho_0-\bar{\rho}, \mathbf{u}_0, \nabla\phi_0, \phi_0^2-1)\in \dot{H}^{-s}_{per}(\mathbb{T}^d)$ for some $s\in [0,\f{d}{2})$, then both \eqref{h124} and \eqref{hg124} still hold,
	and \begin{equation}
	\label{h125*}
	\|\rho(t)-\bar{\rho} \|_{H^{N}}^2+\|\nabla\mathbf{u}(t)\|_{H^{N-1}}^2+\|\nabla \phi(t)\|_{H^{N-1}}^2\leq C_0(1+t)^{-(N-1+s)}.
	\end{equation}
\end{corollary}

 To prove Theorem \ref{theo 2.1}, we will use the energy method  developed by \cite{GW 11}, which relies essentially on  the following  two main steps:

Step 1. Energy estimates at $l-$th level.
\begin{equation}
\label{h117}
\f{d}{dt}\mathcal{E}_l^m(t)+\|\nabla^l (\rho-\bar{\rho})(t)\|^2_{H^{m-l}}+ \|\nabla^{l+1}\mathbf{u}(t)\|^2_{H^{m-l}}+ \|\nabla^{l+1} \phi(t)\|^2_{H^{m+1-l}}\leq 0,
\end{equation}for any $0\leq l< m\leq 3$, where
$$\mathcal{E}_l^m(t) \backsimeq \|\nabla^l(\rho-\bar{\rho}, \mathbf{u})(t)\|^2_{H^{m-l}}+\|\nabla^{l+1} \phi(t)\|^2_{H^{m-1-l}}+\|\phi^2-1\|_{L^2}^2,$$ where $A\backsimeq   B$
means $CA\leq B\leq \f{1}{C}A$
for a generic constant $C>0.$

Step 2. Negative Sobolev norm estimate.
\begin{equation}
\label{h118}
\f{d}{dt}\mathcal{E}_{-s}(t)\lesssim \lf(\|\nabla(\rho,  \mathbf{u})\|_{H^2}^2+\|\nabla\phi\|_{H^2}\rg)\mathcal{E}_{-s}(t),\ \ \ 0 < s\leq\f{3}{2},\end{equation} where
$$\mathcal{E}_{-s}=\|\Lambda^{-s} (\rho-\bar{\rho})\|_{L^2}^2+\|\Lambda^{-s} \mathbf{u}\|_{L^{2}}^2+\|\Lambda^{-s} \nabla\phi\|_{L^2}^2+\|\Lambda^{-s} (\phi^2-1)\|_{L^2}^2. $$

If we prove \eqref{h117}, then it is easy to show  that there exits a  solution of the system \eqref{h21} satisfying \eqref{h18} by the continuation argument of local solution (Subsection 3.1).
Also, by using \eqref{h117} and the  Poincar\'e inequality
\begin{equation}\label{h45}{\Big \|}w-\bbint_{\mathbb{T}^3}w(\mathbf{x})d\mathbf{x}{\Big \|}_{L^p}\lesssim \|\nabla w\|_{L^p},\quad  \text{for all}\,\,\, w\in W_{\mathrm{per}}^{1,p}(\mathbb{T}^3),\,\, 1<p<\infty, \end{equation}
combining \eqref{h118},
Lemma \ref{lem24},  the  differential inequality
\begin{equation}
\label{h119}\f{df(t)}{dt}+c_0(f(t))^{1+\f{1}{l+s}}\leq 0\Rightarrow f(t)\leq \lf(f(0)^{-\f{1}{l+s}}+\f{c_0t}{l+s}\rg)^{-(l+s)}\lesssim (1+t)^{-(l+s)},\end{equation} and
 \begin{equation}\label{hr317}\|\nabla^{k}w\|_{L^2}\lesssim \|\nabla^{k+1}w\|_{L^{2}},\,\, k=1,2, \quad  \text{for all}\,\,\,w\in H^{3}_{\mathrm{per}}(\mathbb{T}^3), \end{equation}
we can get the algebraic decay \eqref{hg124} and \eqref{h125} (Subsection 3.2). Therefore, the estimates \eqref{h117} and \eqref{h118} are essential in the proof of Theorem \ref{theo 2.1}.

Here, we briefly review some  difficulties and key analytical techniques in deriving \eqref{h117} and \eqref{h118}, compared with previous works  in \cite{W 12, TZ 11, GTY 16}.
The main difficulty comes from the Cahn-Hillard  equation $\eqref{h21}_{3,4}$, rewritten it as a fourth order nonlinear PDE
\begin{equation}
\label{h120}
\phi_t+\f{1}{\rho}\Delta \lf(\f{\Delta\phi}{\rho}\rg)- \f{2}{\rho}\Delta\phi+ \f{3}{\rho}\lf(1-\phi^2\rg)\Delta\phi=-\mathbf{u}\cdot\nabla\phi+\f{6\phi}{\rho}(\nabla\phi)^2,\end{equation}where  the strong nonlinear term $\lf(1-\phi^2\rg)\Delta\phi$ makes a trouble for desired estimates  because both of  $\|\phi(t)\|_{L^2}$
and $\|\phi(t)\|_{L^\infty}$ are not small.
 On the other hand, the coupling between the Navier-Stokes equations $\eqref{h21}_{1,2}$ and the Cahn-Hilliard equation  \eqref{h120} also bring trouble to us.
In order to overcome this difficulties, we first
find that $\|\phi^2-1\|_{L^2}$ is small for small initial data and then $\|\phi(t)\|_{L^\infty}$ is bounded (Lemma \ref{lem31}). This makes for us to assume that not only $\| (\sigma,\mathbf{u})\|_{H^{3}}+\| \nabla\phi\|_{H^{2}}$ is small, but also  $\|\phi^2-1\|_{L^2}$ is small in a priori estimates. Based on these facts, we complete the estimate \eqref{hr329} for $\nabla\phi$ at $k-$th level, where a new $L^p-$estimate \eqref{hg21}  for nonlinear functions is used essentially together with Gagliardo-Nirenberg inequality \eqref{hg20}  (Lemma \ref{lem33}). Next, we derive the estimate \eqref{hr316} and \eqref{hr326}  for $(\rho-\bar{\rho}, \mathbf{u})$ at  $k-$th level, where the more delicate estimates are needed because of the dependence on $\phi$ of viscosity coefficients $\nu$ and  $\lambda$ (Lemma \ref{lem32} and Lemma \ref{lem35}). Combining Lemmas  \ref{lem31}-\ref{lem35}, we get \eqref{h117} (see Subsection 3.1).
Another main point in this paper is how to get the negative Sobolev norm estimate \eqref{h118} in the case of periodic boundary problem, which is quite different with Cauchy problem in $\mathbb{R}^3$ of previous works \cite{W 12, TZ 11, GTY 16} due to they are based essentially on the following inequality
$$\|w\|_{L^{\f{3}{s}}}\lesssim \begin{cases}
\displaystyle\|\nabla w\|^{\f{1}{2}+s}\|\nabla^2 w\|^{\f{1}{2}-s}, \quad s\in (0,\f{1}{2}],\\
\displaystyle\| w\|^{s-\f{1}{2}}\|\nabla w\|^{\f{3}{2}-s}, \qquad s\in (\f{1}{2},\f{3}{2}),
\end{cases}$$which does not hold for periodic boundary problem.
 Therefore, in order to estimate  $\|\Lambda^{-s} (\rho-\bar{\rho}, \mathbf{u}, \nabla\phi)\|^2$, we must rely  on Lemma \ref{lem25} and the  Poincar\'e inequality  \eqref{h45} (see Lemma \ref{lem41}). For example, we have
$$
\begin{aligned}
&\int_{\mathbb{T}^3} \Lambda^{-s}\lf[(\mathbf{u}, \nabla)\mathbf{u} \cdot \rg]\Lambda^{-s}\mathbf{u} d\mathbf{x}=\int_{\mathbb{T}^3} (\mathbf{u}', \nabla)\Lambda^{-s} \mathbf{u}\cdot\Lambda^{-s}\mathbf{u} d\mathbf{x}\\
&+\int_{\mathbb{T}^3} \Lambda^{-s}\lf[(\mathbf{u}-\mathbf{u}', \nabla) \mathbf{u}\rg]\cdot\Lambda^{-s}\mathbf{u} d\mathbf{x}=\int_{\mathbb{T}^3} \Lambda^{-s}\lf[(\mathbf{u}-\mathbf{u}', \nabla) \mathbf{u}\rg]\cdot\Lambda^{-s}\mathbf{u} d\mathbf{x}\\&\overset{\eqref{hg25}}\lesssim \|(\mathbf{u}-\mathbf{u}',\nabla)\mathbf{u}\|_{L^{\f{1}{\f{1}{2}+\f{s}{3}}}}\|\Lambda^{-s}\mathbf{u}\|\lesssim \|\mathbf{u}-\mathbf{u}'\|_{L^\f{3}{s}}\|\nabla\mathbf{u}\|\|\Lambda^{-s}\mathbf{u}\|
\\&\overset{\eqref{h45}}\lesssim \|\nabla\mathbf{u}\|_{L^\f{3}{s}}\|\nabla\mathbf{u}\| \|\Lambda^{-s}\mathbf{u}\|
\lesssim \|\nabla\mathbf{u}\|_{H^2}^2 \|\Lambda^{-s}\mathbf{u}\|,
\end{aligned}
$$where
$\displaystyle\mathbf{u}'=\bbint_{\mathbb{T}^3}\mathbf{u} d\mathbf{x}$.
Also, for the estimate on $\|\Lambda^{-s} (\phi^2-1)\|^2$,
multiplying  \eqref{h120} by $2\phi$ and applying $\Lambda^{-s}$  to the resulting equality \eqref{hr330}, we could obtain the following type of inequality
$$\f{d}{dt}\|\Lambda^{-s} \lf(\phi^2-1\rg)\|^2+\|\Lambda^{-s} \nabla\lf(\phi^2-1\rg)\|^2+\|\Lambda^{-s} \Delta\lf(\phi^2-1\rg)\|^2\lesssim \cdots \cdots, $$which is another key point (see Lemma \ref{lem42}).


\section{The local existence and some energy estimates}
\setcounter{equation}{0}
In this section,  a priori estimate   for \eqref{h21} are established. First, we will  give  the solution space.  For any interval $I\subset[0,\infty)$, and $\forall M>0$,   we suppose  that
$(\sigma, \mathbf{u}, \phi)\in X_{M}(I)$ is the solution to the system
\eqref{h21}, where the solution space $X_{M}(I)$  is defined as follows
\begin{equation}\label{h310}
\begin{aligned}
X_{M}(I)&\overset{\mathrm{def}}=\Big\{(\sigma,\mathbf{u}, \phi)\Big|\,(\sigma,\mathbf{u})\in C(\emph{I};H_{\mathrm{per}}^3(\mathbb{T}^3)),\,\,\,\nabla\phi \in C(I;H_{\mathrm{per}}^2(\mathbb{T}^3)),\\
&\qquad \phi^2-1\in C(I;L_{\mathrm{per}}^{2}(\mathbb{T}^3)),\,\,\,\nabla\sigma \in L^2(I;H_{\mathrm{per}}^2(\mathbb{T}^3)),
\\
&\qquad\nabla\phi \in L^2(I;H_{\mathrm{per}}^4(\mathbb{T}^3)),\,\,\,\nabla \mathbf{u}\in L^2(I;H_{\mathrm{per}}^3(\mathbb{T}^3)),\\
&\quad\sup_{ t \in I}\lf(\| (\sigma,\mathbf{u})\|_{H^{3}}+\| \nabla\phi\|_{H^{2}}+\|\phi^2(t)-1\|\rg)\leq M,\ \inf_{\mathbf{x}\in \mathbb{T}^3, t\in I}\rho(\mathbf{x},t)>0
\Big\}.
\end{aligned}\end{equation}
Now the local existence for \eqref{h21} is established below. The proof process is classical, and it can be easily obtained by using traditional methods such as linearization techniques for equations and fixed point theorems. We will not give the details.
\begin{proposition}  \label{pro311}
	\textbf{(local existence). } Suppose that \eqref{h22}-\eqref{hhh} and \eqref{A1} are satisfied, $\displaystyle\bar{\rho}=\bbint_{\mathbb{T}^3}\rho_0(\mathbf{x})d\mathbf{x}$. Let  $\|(\rho_0-\bar{\rho}, \mathbf{u}_0)\|_{H^{3}} +\| \nabla\phi_0\|_{H^{2}}+\|\phi_0^2-1\|\leq M$, $\displaystyle\inf\limits_{\mathbf{x}\in \mathbb{T}^3 }\rho_0(\mathbf{x})>0$, then there exists $T^*$ small enough, such that, there exists a unique solution $(\rho,\mathbf{u}, \phi)\in X_{2M}\big([0,T^*]\big)$ to the system \eqref{h21}, satisfying
\begin{eqnarray}
  &&(\rho, \mathbf{u},\phi) \in C([0,T^*];H^3(\mathbb{T}^3)), \rho\in L^2([0,T^*];H^{3}(\mathbb{T}^3)), \notag\\
  &&\nabla\phi \in L^2([0,T^*];H^4(\mathbb{T}^3)),\ \nabla \mathbf{u}\in L^2([0,T^*];H^3(\mathbb{T}^3)).
\end{eqnarray}
\end{proposition}

\vspace{3ex}On the basis of the existence of local solutions, We will now extend the local solution to the global solution.  For this purpose, we need to give some related energy inequalities.
Setting
\begin{equation}\label{sigma}
  \sigma=\rho-\bar{\rho},
\end{equation}
we rewrite $\eqref{h21}_{1}-\eqref{h21}_{2}$ and $\eqref{h21}_{3,4}$  as following
\begin{equation}\label{h313}\begin{cases}
\begin{aligned}
&\sigma_t+\bar{\rho}{\rm div} \mathbf{u}=g_1,\\
&\mathbf{u}_t-\f{2}{\bar{\rho}}{\rm div}\lf[\nu(\phi)\mathbf{D}(\mathbf{u})\rg]-\f{1}{\bar{\rho}}\nabla\lf[(\nu(\phi)+\lambda(\phi)) {\rm div}\mathbf{u}\rg]+\f{p'(\bar{\rho})}{\bar{\rho}}\nabla\sigma+\frac{1}{\bar\rho}\nabla\phi\Delta\phi=\mathbf{g}_2,\end{aligned}
\end{cases}
\end{equation}
and
\begin{equation}\label{h3280} \phi_t+ \mathbf{u}\cdot\nabla\phi=-\f{1}{\rho}\Delta \lf(\f{\Delta\phi}{\rho}\rg)+ \f{3\phi^2-1}{\rho}\Delta\phi+\f{6\phi}{\rho}(\nabla\phi)^2,\end{equation}
where
\begin{equation}\label{h314}g_1=-{\rm div}(\sigma\mathbf{u}),
\end{equation}
\begin{equation}\label{h315}\begin{aligned}
\displaystyle\mathbf{g}_2=&\displaystyle-
(\mathbf{u},\nabla)\mathbf{u}+h_1(\sigma)\nabla\sigma-2h_2(\sigma){\rm div}\lf[\nu(\phi)\mathbf{D}(\mathbf{u})\rg]\\
&\displaystyle-h_2(\sigma)\nabla\lf[(\nu(\phi)+\lambda(\phi)) {\rm div}\mathbf{u}-\nabla\phi\Delta\phi\rg],
\end{aligned}
\end{equation}
and
\begin{equation*}\label{h1h2}
   h_1(\sigma)=\f{p'(\bar{\rho})}{\bar{\rho}}-\f{p'(\rho)}{\rho},\quad
h_2(\sigma)=\f{1}{\bar{\rho}}-\f{1}{\rho}.
\end{equation*}
 By using the solution space \eqref{h310}, combining with the Sobolev embedding theorem, we can choose $M_0>0$, such that, $\forall0<M<M_0$,
\begin{equation}\label{hg313}
0<\f{\bar{\rho}}{2}\leq\rho(\mathbf{x},t)\leq 2\bar{\rho}.
\end{equation}

We first give the following energy estimate:
\begin{lemma}\label{lem31}  Under the  assumption \eqref{h310}, it holds   that
\begin{equation}\label{h38} \f{d}{dt}\int_{\mathbb{T}^3}\lf(\f{\rho}{2}\mathbf{u}^2+G(\rho)+\frac{1}{2}|\nabla
\phi|^2+\frac{\rho}{4}(\phi^2-1)^2\rg)d\mathbf{x} +\f{\nu_0}{2}\|\nabla\mathbf{u}\|^2+\|\nabla
\mu \|^2\leq 0,
\end{equation}	
	\begin{equation}
	\label{hr33}
	\|\phi(t)\|\lesssim 1\quad \text{and}\quad
\|\phi(t)\|_{L^\infty}\lesssim 1,	\end{equation} where \begin{equation}\label{h25}
G(\rho)=\rho\int_{\bar{\rho}}^\rho\f{p(z)-p(\bar{\rho})}{z^2}dz, \quad\rho>0.
\end{equation}
\end{lemma}
\begin{proof}
Using \eqref{h25} and
$\eqref{h21}_1$, it is easy to see
\begin{equation}\label{h32}\begin{aligned}
&\rho G'(\rho)=G(\rho)+(p(\rho)-p(\bar{\rho})),\quad \rho G''(\rho)=p'(\rho),\\
&G(\rho)_{t}+{\rm div}(G(\rho) \mathbf{u})+(p(\rho)-p(\bar{\rho})){\rm div}\mathbf{u}=
0.
\end{aligned}\end{equation}
Using $\eqref{h21}_2$ and \eqref{h32} yields
that
\begin{equation}\label{h36}\begin{aligned}
&\f{d}{dt}\int_{\mathbb{T}^3}\lf(\f{1}{2}\rho
\mathbf{u}^2+G(\rho)\rg)d\mathbf{x}+\f{1}{2}\int_{\mathbb{T}^3}\nu(\phi)|\nabla \mathbf{u}+\nabla \mathbf{u}^T|^2d\mathbf{x}\\
&+\int_{\mathbb{T}^3}\lambda(\phi) |{\rm div}\mathbf{u}|^2d\mathbf{x}=-\int_{\mathbb{T}^3}\mathbf{u}\cdot\nabla\phi\Delta\phi d\mathbf{x} .\end{aligned}
\end{equation}
Multiplying $\eqref{h21}_3$ by $\mu$ and using $\eqref{h21}_4$, we have
\begin{equation}\label{h37}\f{1}{2}\f{d}{dt}\int_{\mathbb{T}^3}|\nabla\phi|^2d\mathbf{x}+\f{1}{4}\f{d}{dt}\int_{\mathbb{T}^3}\rho\lf(\phi^2-1\rg)^2d\mathbf{x} +\int_{\mathbb{T}^3}|\nabla\mu |^2d\mathbf{x}=-\int_{\mathbb{T}^3} \nabla(\mathbf{u}\cdot \nabla\phi)\cdot  \nabla\phi d\mathbf{x},
\end{equation}where we used
$$\f{1}{4}\int_{\mathbb{T}^3}\lf(\phi^2-1\rg)^2{\rm div}(\rho \mathbf{u})d\mathbf{x}=-\int_{\mathbb{T}^3}\rho \mathbf{u}\cdot \nabla\phi \lf(\phi^3-\phi\rg)d\mathbf{x}.$$
Adding \eqref{h36} and \eqref{h37}, and using \eqref{h22},
we get \eqref{h38}.
By \eqref{h25}, \eqref{h23} and \eqref{hg313}, we have $G(\bar{\rho})=G'(\bar{\rho})=0$ and
\begin{equation}\label{hg314}
c_{\bar{\rho}}(\rho-\bar{\rho})^2\leq G(\rho)\leq C_{\bar{\rho}}(\rho-\bar{\rho})^2.\end{equation}
Integrating \eqref{h38} for $t$, ans using \eqref{hg313} and \eqref{hg314}, we obtain
\begin{equation}
\label{hr32}
\begin{aligned}
&\|(\rho-\bar{\rho}, \mathbf{u}, \nabla\phi)(t)\|^2+\|\phi^2(t)-1\|^2+\int_0^t \|(\nabla \mathbf{u}, \nabla\mu)(\tau)\|^2d\tau
\\
&\lesssim
\|(\rho_0-\bar{\rho}, \mathbf{u}_0, \nabla\phi_0)\|^2+\|\phi^2_0-1\|^2.\end{aligned}\end{equation}
Also, by  \eqref{hr32}, we have
$$
\|\phi(t)\|^2=\int_{\mathbb{T}^3}\lf(\phi^2(t)-1+1\rg)d\mathbf{x}\leq |\mathbb{T}^3|^\f{1}{2}\lf(\|\phi^2(t)-1\|+|\mathbb{T}^3|^\f{1}{2}\rg)\leq C, $$and by Gagliardo-Nirenberg inequality, we get
$$
\|\phi(t)\|_{L^\infty}\leq C_0\|\nabla^2\phi(t)\|^{\f{3}{4}}\|\phi(t)\|^{\f{1}{4}} +C_0\|\phi(t)\|\leq C.
$$
The proof of Lemma \ref{lem31} is completed.\end{proof}

\vspace{0.2cm}Next, we derive the following estimate on $\nabla \phi$, which plays an essential role in the proof of  Theorem \ref{theo 2.1}.

\begin{lemma}\label{lem33}  Under the  assumption \eqref{h310}, it holds   that
	\begin{equation}
	\label{hr329}
	\f{d}{dt}\|\nabla^{k+1} \phi\|^2+\f{1}{4\bar{\rho}^2}\|\nabla^{k+3}\phi\|^2
	\lesssim \eta_1\|\nabla^{k+1}
	\sigma\|^2,\ \ \mathrm{for}\  k=0,1, 2.
	\end{equation}
\end{lemma}
\begin{proof}
Applying $\nabla^k$ to \eqref{h3280} and
multiplying it by $-\Delta \nabla^k \phi$ yields that
\begin{equation}\label{h328}
\begin{aligned}
&\f{1}{2}\f{d}{dt}\int_{\mathbb{T}^3}|\nabla^{k+1} \phi|^2d\mathbf{x} +\int_{\mathbb{T}^3}{\Big|}\nabla \lf(\f{\nabla^k \Delta\phi }{\rho}\rg){\Big|}^2d\mathbf{x}-\int_{\mathbb{T}^3} \nabla^k (\mathbf{u}\cdot\nabla\phi) \Delta \nabla^k\phi d\mathbf{x}\\
&=\sum_{1\leq l\leq k}
C_{k}^{l}\int_{\mathbb{T}^3} \nabla^{l}\lf(\f{1}{\rho}\rg) \nabla^{k-l}\Delta \lf(\f{\Delta\phi}{\rho}\rg) \nabla^{k}\Delta\phi d\mathbf{x}\\
&+\sum_{1\leq l\leq k}
C_{k}^{l}\int_{\mathbb{T}^3}\Delta\lf[ \nabla^{l}\lf(\f{1}{\rho}\rg) \nabla^{k-l}\Delta\phi\rg]\f{\nabla^{k}\Delta\phi}{\rho}  d\mathbf{x}-\int_{\mathbb{T}^3} \f{3\phi^2-1}{\rho} | \nabla^{k}\Delta \phi|^2d\mathbf{x}\\
&-\sum_{1\leq l\leq k}
C_{k}^{l}\int_{\mathbb{T}^3} \nabla^{l}\lf(\f{3\phi^2-1}{\rho}\rg) \nabla^{k-l}\Delta \phi \nabla^{k}\Delta\phi d\mathbf{x}-\int_{\mathbb{T}^3} \nabla^{k}\lf(\f{6\phi}{\rho}(\nabla\phi)^2\rg)\Delta \nabla^{k}\phi d\mathbf{x}.
\end{aligned}
\end{equation}
Noticing that
$$\begin{aligned}
&\int_{\mathbb{T}^3}{\Big|}\nabla \lf(\f{\nabla^{k} \Delta\phi }{\rho}\rg){\Big|}^2d\mathbf{x}\\
&=\int_{\mathbb{T}^3}\f{1}{\rho^2}|\nabla^{k+1} \Delta\phi|^2d\mathbf{x}+\int_{\mathbb{T}^3}\f{|\nabla\rho|^2}{\rho^4}| \nabla^{k} \Delta\phi|^2d\mathbf{x}-\int_{\mathbb{T}^3} \f{3\phi^2-1}{\rho} | \nabla^{k}\Delta \phi|^2d\mathbf{x}\\
&\leq -\int_{\mathbb{T}^3} \f{\phi^2-1}{\rho} | \nabla^{k}\Delta \phi|^2d\mathbf{x},
\end{aligned}$$
and using \eqref{hg313}, we obtain  from \eqref{h328}
\begin{equation}\label{hg328}
\f{1}{2}\f{d}{dt}\|\nabla^{k+1}\phi\|^2+\f{1}{4\bar{\rho}^2}\|\nabla^{k+1}\Delta \phi\|^2
\leq I_{1},
\end{equation}
where
\begin{equation}\label{h329}\begin{aligned}
I_{1}&=-\underline{\int_{\mathbb{T}^3} \f{\phi^2-1}{\rho} |\nabla^k\Delta \phi|^2d\mathbf{x}}_{I_{1}^1}+\underline{\int_{\mathbb{T}^3} \nabla^k (\mathbf{u}\cdot\nabla\phi) \Delta \nabla^k\phi d\mathbf{x}}_{I_{1}^2}\\&+\underline{\sum_{1\leq l\leq k}
	C_{k}^{l}\int_{\mathbb{T}^3} \nabla^{l}\lf(\f{1}{\rho}\rg) \nabla^{k-l}\Delta \lf(\f{\Delta\phi}{\rho}\rg) \nabla^{k}\Delta\phi d\mathbf{x}}_{I_{1}^3}\\
&+\underline{\sum_{1\leq l\leq k}
	C_{k}^{l}\int_{\mathbb{T}^3}\Delta\lf[ \nabla^{l}\lf(\f{1}{\rho}\rg) \nabla^{k-l}\Delta\phi\rg]\f{\nabla^{k}\Delta\phi}{\rho}  d\mathbf{x}}_{I_{1}^4}\\
&-\underline{\sum_{1\leq l\leq k}
	C_{k}^{l}\int_{\mathbb{T}^3}  \nabla^{l}\lf(\f{3\phi^2-1}{\rho}\rg) \nabla^{k-l}\Delta \phi \nabla^{k}\Delta\phi d\mathbf{x}}_{I_{1}^5}\\
&-\underline{\sum_{0\leq l\leq k}
	C_{k}^{l}\int_{\mathbb{T}^3} \nabla^l\lf(\f{6\phi}{\rho}\rg)\nabla^{k-l}(\nabla\phi)^2\Delta \nabla^{k}\phi d\mathbf{x}}_{I_{1}^6}.\end{aligned}
\end{equation}
We estimate  $I_{1}$. For  $I_{1}^1$ and $I_{1}^2$, by using H\"oder inequality and \eqref{h310}, we have
$$
\begin{aligned}
&\displaystyle\begin{aligned}I_{1}^1
&\displaystyle\lesssim\|\phi^2-1\|\|\Delta \nabla^{k}\phi\|^2_{L^4}\lesssim\eta_1\|\nabla^{k+2}\phi\|^2_{L^4}\\
&\displaystyle\overset{\eqref{hg20}}\lesssim\eta_1\|\nabla^{k+3} \phi\|^{2-\f{3}{2}}\|\nabla^{k+2} \phi\|^\f{3}{2}\lesssim\eta_1\|\nabla^{k+3} \phi\|^2,\end{aligned}\\&
\begin{aligned}I_{1}^2
&\displaystyle\lesssim\|\nabla^k (\mathbf{u}\cdot\nabla\phi)\|\|\Delta \nabla^{k}\phi\|\\
&\displaystyle\overset{\eqref{hg22}}\lesssim\lf(\|\mathbf{u}\|_{L^\infty}\|\nabla^{k+1} \phi\|+\|\nabla \phi\|_{L^\infty}\|\nabla^{k}\mathbf{u}\|\rg)\|\nabla^{k+2} \phi\|\\
&\displaystyle\lesssim\eta_1\|\nabla^{k+3} \phi\|^2.\end{aligned}
\end{aligned}
$$
Also, for $I_{1}^3$, by using H\"oder inequality and \eqref{h310}, we have
$$
\begin{aligned}
I^3_{1}&=-\sum_{1\leq l\leq k}
C_{k}^{l}\int_{\mathbb{T}^3} \nabla\lf[\nabla^{l}\lf(\f{1}{\rho}\rg) \nabla^{k}\Delta\phi \rg] \cdot \nabla^{k-l}\nabla\lf[{\rm div}\lf(\f{\nabla\phi}{\rho}\rg)-\nabla\lf(\f{1}{\rho}\rg)\cdot\nabla\phi\rg]  d\mathbf{x}\\
&\begin{aligned}\lesssim \sum_{1\leq l\leq k}&\lf({\Big\|}\nabla^{l}\lf(\f{1}{\rho}\rg){\Big\|}_{L^\infty} \|\nabla^{k+3}\phi\|+\Big\|\nabla^{l+1}\lf(\f{1}{\rho}\rg){\Big\|}_{L^3} \|\nabla^{k+2}\phi\|_{L^6}\rg)\times\\
&\times \lf({\Big\|}\nabla^{k-l+2}\lf(\f{\nabla\phi}{\rho}\rg){\Big\|}+{\Big\|}\nabla^{k-l+1}\lf[\nabla\lf(\f{1}{\rho}\rg)\cdot\nabla\phi\rg]{\Big\|}\rg)\end{aligned}
\\
&\begin{aligned}\overset{\eqref{hg21},\, \eqref{hg22}}\lesssim &\sum_{1\leq l\leq k}\lf(\|\nabla^{l}\sigma\|_{L^\infty} \|\nabla^{k+3}\phi\|+\|\nabla^{l+1}\sigma\|_{L^3} \|\nabla^{k+2}\phi\|_{L^6}\rg)\times\\
&\times (\|\nabla^{k-l+3}\phi\|+\|\nabla\phi\|_{L^\infty}\|\nabla^{k-l+2}\sigma\|+\|\nabla^{k-l+2}\phi\|)\end{aligned}\\
&\overset{\eqref{hg20},\, \eqref{h310}}\lesssim \eta_1\sum_{1\leq l\leq k} \|\nabla^{k+3}\phi\|(\|\nabla^{k-l+3}\phi\|+\|\nabla^{k-l+2}\sigma\|+\|\nabla^{k-l+2}\phi\|)\\
&\overset{\eqref{hr317}}\lesssim \eta_1\lf(\|\nabla^{k+1} \sigma\|^2+\|\nabla^{k+3} \phi\|^2\rg).
\end{aligned}$$
By the similar arguments, we have
$$
\begin{aligned}
I^4_{1}&=-\sum_{1\leq l\leq k}
C_{k}^{l}\int_{\mathbb{T}^3}\nabla\lf[ \nabla^{l}\lf(\f{1}{\rho}\rg) \nabla^{k-l}\Delta\phi\rg]\cdot \nabla\lf[\f{\nabla^{k}\Delta\phi}{\rho} \rg] d\mathbf{x}\\
&\begin{aligned}\lesssim \sum_{1\leq l\leq k}&\lf({\Big\|}\nabla^{l}\lf(\f{1}{\rho}\rg){\Big\|}_{L^\infty} \|\nabla^{k-l+3}\phi\|+\Big\|\nabla^{l+1}\lf(\f{1}{\rho}\rg){\Big\|}_{L^3} \|\nabla^{k-l+2}\phi\|_{L^6}\rg)\times\\
&\times \lf(\|\nabla^{k+3}\phi\|+\|\nabla^{k+2}\phi\|{\Big\|}\nabla\lf(\f{1}{\rho}\rg){\Big\|}_{L^\infty}\rg)\end{aligned}
\\
&\overset{\eqref{hg21}}\lesssim \sum_{1\leq l\leq k}\lf(\|\nabla^{l}\sigma\|_{L^\infty}\|\nabla^{k-l+3}\phi\|+\|\nabla^{l+1}\sigma\|_{L^3} \|\nabla^{k-l+2}\phi\|_{L^6}\rg)\times\\
&\times  (\|\nabla^{k+3}\phi\|+\|\nabla^{k+2}\sigma\|\|\nabla\sigma\|_{L^\infty})\\
&\overset{\eqref{hg20},\, \eqref{h310}}\lesssim \eta_1\sum_{1\leq l\leq k} \|\nabla^{k-l+3}\phi\|(\|\nabla^{k+3}\phi\|+\|\nabla^{k+2}\sigma\|)\overset{\eqref{hr317}}\lesssim \eta_1\lf(\|\nabla^{k+1} \sigma\|^2+\|\nabla^{k+3} \phi\|^2\rg),
\end{aligned}$$
 $$
\begin{aligned}
I^5_{1}&=\sum_{1\leq l\leq k}
C_{k}^{l}\int_{\mathbb{T}^3}  \nabla^{l}\lf(\f{3\phi^2-1}{\rho}\rg) \nabla^{k-l}\Delta \phi \nabla^{k}\Delta\phi d\mathbf{x}\\
&\lesssim \sum_{1\leq l\leq k}\lf({\Big\|}\nabla^{l}\lf(\f{\phi^2}{\rho}\rg){\Big\|}+\Big\|\nabla^{l}\lf(\f{1}{\rho}\rg){\Big\|}\rg)\|\nabla^{k-l+2}\phi\|_{L^3} \|\nabla^{k+2}\phi\|_{L^6}
\\
&\overset{\eqref{hg21}}\lesssim \sum_{1\leq l\leq k}\lf(\|\nabla^{l}\sigma\|+\|\nabla^{l}\phi\|\rg)\|\nabla^{k-l+2}\phi\|_{L^3} \|\nabla^{k+2}\phi\|_{L^6}\\
&\overset{\eqref{hg20},\, \eqref{h310}}\lesssim \eta_1\sum_{1\leq l\leq k} \lf(\|\nabla^{l}\sigma\|+\|\nabla^{l}\phi\|\rg)\|\nabla^{k+3}\phi\|\overset{\eqref{hr317}}\lesssim \eta_1\lf(\|\nabla^{k+1} \sigma\|^2+\|\nabla^{k+3} \phi\|^2\rg),
\end{aligned}$$ and $$
\begin{aligned}
I^6_{1}&=\sum_{0\leq l\leq k}
C_{k}^{l}\int_{\mathbb{T}^3} \nabla^l\lf(\f{6\phi}{\rho}\rg)\nabla^{k-l}(\nabla\phi)^2\Delta \nabla^{k}\phi d\mathbf{x}\\
&\lesssim \sum_{1\leq l\leq k}{\Big\|}\nabla^{l}\lf(\f{\phi}{\rho}\rg){\Big\|}_{L^3}\|\nabla^{k-l}(\nabla\phi)^2\| \|\nabla^{k+2}\phi\|_{L^6}
\\
&\overset{\eqref{hg21}, \eqref{hg22}}\lesssim \sum_{1\leq l\leq k}\lf(\|\nabla^{l}\sigma\|_{L^3}+\|\nabla^{l}\phi\|_{L^3}\rg)\|\nabla \phi\|_{L^\infty}\|\nabla^{k-l+1}\phi\| \|\nabla^{k+2}\phi\|_{L^6}\\
&\overset{\eqref{hg20},\, \eqref{h310}}\lesssim \eta_1\sum_{1\leq l\leq k} \|\nabla^{k-l+1}\phi\|\|\nabla^{k+3}\phi\|\overset{\eqref{hr317}}\lesssim \eta_1\|\nabla^{k+3} \phi\|^2.
\end{aligned}$$
Substituting the estimates for $I_{1}^j(j=1,\cdots,6)$ into \eqref{h329}, we have
$$I_{1}\lesssim \eta_1\lf(\|\nabla^{k+1}
\sigma\|^2+
\|\nabla^{k+3}\phi\|^2\rg),$$
and then, we obtain
\eqref{hr329} from \eqref{hg328}.
The proof of Lemma \ref{lem33} is completed.\end{proof}

\vspace{0.2cm} In the following two lemmas, we will give the estimates for $(\sigma, \mathbf{u})$.
\begin{lemma}\label{lem32}  Under the  assumption \eqref{h310}, it holds   that
	\begin{equation}
	\label{hr316}
		\f{d}{dt}\lf(\|\nabla^k\mathbf{u}\|^2
	+\f{p'(\bar{\rho})}{\bar{\rho}^2}\|\nabla^k\sigma\|^2\rg)+\f{\nu_0}{\bar{\rho}}\|\nabla^{k+1} \mathbf{u}\|^2\lesssim \eta_1\lf(\|\nabla^{k}
	\sigma\|^2+\|\nabla^{k+2}\phi\|^2\rg),
	\end{equation}for $k=0,\cdots, 3$.
\end{lemma}
\begin{proof} We first derive \eqref{hr316} for $k=0$. Multiplying $\eqref{h313}_2$ by $\mathbf{u}$, and using $\eqref{h313}_1$ and \eqref{h22}, we have
\begin{equation}\label{hh324}\begin{aligned}
\f{1}{2}&\f{d}{dt}\lf(\|\mathbf{u}\|^2
+\f{p'(\bar{\rho})}{\bar{\rho}^2}\|\sigma\|^2\rg)+\f{\nu_0}{\bar{\rho}}\int_{\mathbb{T}^3}|\mathbf{D}(\mathbf{u})|^2 d\mathbf{x}\\
&\leq \underline{\int_{\mathbb{T}^3} \lf(\mathbf{g}_2\cdot\mathbf{u}-\f{p'(\bar{\rho})}{\bar{\rho}^2} g_1\sigma\rg) d\mathbf{x}}_{I_{2}}-\underline{\frac{1}{\bar\rho}\int_{\mathbb{T}^3} \nabla\phi\Delta\phi \cdot \mathbf{u}d\mathbf{x}}_{I_{3}}.
\end{aligned}
\end{equation}
Using \eqref{h314} and \eqref{h315}, it is easy to check by the same lines as in \cite[Lemma 2.1]{W 12} that
$$
I_2\lesssim \eta_1\lf(\|\sigma\|^2+\|\nabla \mathbf{u}\|^2\rg).
$$
Also, for $I_3$, we have
$$
I_3\lesssim \|\nabla\phi\|_{L^3}\|\Delta\phi\|\|\mathbf{u}\|_{L^6}\overset{\eqref{hg20}, \eqref{h310}}\lesssim \eta_1\lf(\|\nabla^2\phi\|^2+\|\nabla \mathbf{u}\|^2\rg).
$$
Therefore, we obtain \eqref{hr316} for $k=0$ from \eqref{hh324}.

\vspace{0.2cm}Next, we  derive \eqref{hr316} for $k=1,\cdots, 3$.  Applying $\nabla^k$ to \eqref{h313},  we have
\begin{equation}\label{h317}\begin{aligned}
&\nabla^k \sigma_{t}+\bar{\rho}{\rm div}\nabla^k \mathbf{u}=\nabla^k g_1, \\
&\begin{aligned}\nabla^k \mathbf{u}_t&-\f{2}{\bar{\rho}}{\rm div}\lf[\nu(\phi)\mathbf{D}(\nabla^k\mathbf{u})\rg]-\f{1}{\bar{\rho}}\nabla\lf[(\nu(\phi)+\lambda(\phi)) {\rm div}\nabla^k\mathbf{u}\rg]+\f{p'(\bar{\rho})}{\bar{\rho}}\nabla^{k}\nabla\sigma\\
&+\frac{1}{\bar\rho}\nabla^k\lf[\nabla\phi\Delta\phi\rg]=\nabla^k\mathbf{g}_2+\frac{2}{\bar{\rho}}{\rm div}\lf[\nabla^{k}(\nu(\phi)\mathbf{D} (\mathbf{u}))-\nu(\phi)\mathbf{D}(\nabla^k\mathbf{u})\rg]\\
&+\frac{1}{\bar{\rho}}\nabla\lf[\nabla^{k}((\nu(\phi)+\lambda(\phi)){\rm div} \mathbf{u})-(\nu(\phi)+\lambda(\phi)) {\rm div}\nabla^k\mathbf{u}\rg].\end{aligned}
\end{aligned}
\end{equation}
Multiplying $\eqref{h317}_2$ by $\nabla^{k}
\mathbf{u}$, and using $\eqref{h317}_1$, \eqref{h314} and \eqref{h315}, we have
\begin{equation}\label{h324}\begin{aligned}
\f{1}{2}&\f{d}{dt}\lf(\|\nabla^{k}\mathbf{u}\|^2
+\f{p'(\bar{\rho})}{\bar{\rho}^2}\|\nabla^{k}\sigma\|^2\rg)+\f{1}{\bar{\rho}}\int_{\mathbb{T}^3}\nu(\phi)|\mathbf{D}(\nabla^{k}\mathbf{u})|^2 d\mathbf{x}\\
&+\f{1}{\bar{\rho}}\int_{\mathbb{T}^3}(\nu(\phi)+\lambda(\phi))|{\rm div}\nabla^{k}\mathbf{u}|^2 d\mathbf{x}
=I_{4},
\end{aligned}
\end{equation}
where
\begin{equation}\label{h327}\begin{aligned}
I_{4}=&\f{P'(\bar{\rho})}{\bar{\rho}^2}\underline{\int_{\mathbb{T}^3} \nabla^{k}\sigma \nabla^{k}{\rm div}(\sigma\mathbf{u})d\mathbf{x}}_{I_{4}^1}+\underline{\int_{\mathbb{T}^3} \nabla^{k}\lf[-
	(\mathbf{u},\nabla)\mathbf{u}+h_1(\sigma)\nabla\sigma\rg] \cdot \nabla^{k}\mathbf{u}d\mathbf{x}}_{I_{4}^2}\notag \\
&\underline{-\int_{\mathbb{T}^3} \nabla^{k}\lf[2h_2(\sigma){\rm div}\lf(\nu(\phi)\mathbf{D}(\mathbf{u})\rg)+h_2(\sigma)\nabla\lf((\nu(\phi)+\lambda(\phi)) {\rm div}\mathbf{u}\rg)\rg] \cdot \nabla^{k}\mathbf{u}d\mathbf{x}}_{I_{4}^3}\notag
\end{aligned}
\end{equation}
\begin{equation}\label{h327}\begin{aligned}
&\underline{-\frac{1}{\bar\rho}\int_{\mathbb{T}^3} \nabla^{k}\lf[\nabla\phi\Delta\phi\rg]\cdot \nabla^{k}\mathbf{u}d\mathbf{x}}_{I_{4}^4}+\f{2}{\bar{\rho}}\underline{\sum_{1\leq l\leq k}
	C_{k}^{l}\int_{\mathbb{T}^3} \nabla^{l}\nu(\phi)\mathbf{D} (\nabla^{k-l}\mathbf{u}):
	\nabla \nabla^{k}\mathbf{u} d\mathbf{x}}_{I_{4}^5}\\
&+\f{1}{\bar{\rho}}\underline{\sum_{1\leq l\leq k}
	C_{k}^{l}\int_{\mathbb{T}^3} \nabla^{l}(\nu(\phi)+\lambda(\phi)){\rm div} \nabla^{k-l}\mathbf{u}\,{\rm div} \nabla^{k}\mathbf{u} d\mathbf{x}}_{I_{4}^6}\\
&+\underline{\int_{\mathbb{T}^3} \nabla^{k}\lf[h_2(\sigma)\nabla\phi\Delta\phi\rg]\cdot \nabla^{k}\mathbf{u}d\mathbf{x}}_{I_{4}^7}.
\end{aligned}
\end{equation}
We   estimate  $I_{4}$.  Noticing that
$$
I_4^1=\f{1}{2}\int_{\mathbb{T}^3} {\rm div}\mathbf{u} (\nabla^{k}\sigma)^2d\mathbf{x}
+\sum_{1\leq l\leq k}
C_{k}^{l}\int_{\mathbb{T}^3} {\rm div}\lf(\nabla^{l}\mathbf{u}\nabla^{k-l}\sigma\rg)
\nabla^{k}\sigma d\mathbf{x},$$
we have
$$
\begin{aligned}
I_4^1&\lesssim \| {\rm div}\mathbf{u}\|_{L^\infty}\|\nabla^{k}\sigma\|^2+\lf(\|\nabla^{2}\mathbf{u}\|_{L^6}\|\nabla^{k-1}\sigma\|_{L^3}+\|\nabla\mathbf{u}\|_{L^\infty}\|\nabla^{k}\sigma\|\rg)\|\nabla^{k}\sigma\|\\
&+\sum_{2\leq l\leq k}
\lf(\|\nabla^{l+1}\mathbf{u}\|\|\nabla^{k-l}\sigma\|_{L^\infty}+\|\nabla^{l}\mathbf{u}\|_{L^6}\|\nabla^{k-l+1}\sigma\|_{L^3}\rg)\|\nabla^{k}\sigma\|\\
&\overset{\eqref{h310}, \eqref{hg20}}\lesssim \eta_1\lf(\|\nabla^{k}\sigma\|^2+\|\nabla^{k-1}\sigma\|^2\rg) +\eta_1\sum_{2\leq l\leq k}\|\nabla^{l+1} \mathbf{u}\|^2\\
&\overset{\eqref{hr317}}\lesssim \eta_1\lf(\|\nabla^{k}\sigma\|^2+\|\nabla^{k+1} \mathbf{u}\|^2\rg).\end{aligned}$$
Next, for  $I_{4}^2$, we have
$$
\begin{aligned}
I_4^2&=-\int_{\mathbb{T}^3} \nabla^{k-1}\lf[-
(\mathbf{u},\nabla)\mathbf{u}+h_1(\sigma)\nabla\sigma\rg] \cdot \nabla^{k+1}\mathbf{u}d\mathbf{x}\\
&\overset{\eqref{hg22}}\lesssim \lf(\|\mathbf{u}\|_{L^\infty}\|\nabla^{k}\mathbf{u}\|+\|\nabla^{k-1}\mathbf{u}\|\|\nabla\mathbf{u}\|_{L^\infty}\rg)\|\nabla^{k+1}\mathbf{u}\|\\
&+\lf(\|h_1(\sigma)\|_{L^\infty}\|\nabla^{k}\sigma\|+\|\nabla^{k-1}h_1(\sigma)\|\|\nabla\sigma\|_{L^\infty}\rg)\|\nabla^{k+1}\mathbf{u}\|\\
&\overset{\eqref{h310}, \eqref{hg21}}\lesssim \eta_1\lf(\|\nabla^{k}\mathbf{u}\|+\|\nabla^{k-1}\mathbf{u}\|+\|\nabla^{k}\sigma\|+\|\nabla^{k-1}\sigma\|\rg)\|\nabla^{k+1}\mathbf{u}\| \\
&\overset{\eqref{hr317}}\lesssim \eta_1\lf(\|\nabla^{k}\sigma\|^2+\|\nabla^{k+1} \mathbf{u}\|^2\rg).\end{aligned}$$
For the term $I_{4}^3$,  using   Leibniz formula and  H\"oder inequality yields that \begin{equation}\label{hr327}\begin{aligned}
I_{4}^3&=\int_{\mathbb{T}^3} \nabla^{k-1}\lf[2h_2(\sigma){\rm div}\lf(\nu(\phi)\mathbf{D}(\mathbf{u})\rg)+h_2(\sigma)\nabla\lf((\nu(\phi)+\lambda(\phi)) {\rm div}\mathbf{u}\rg)\rg] \cdot \nabla^{k+1}\mathbf{u}d\mathbf{x}\\
&=\underline{ \sum_{0\leq l\leq k-1}
C_{k-1}^{l}\int_{\mathbb{T}^3} \nabla^lh_2(\sigma) {\rm div} \nabla^{k-l-1}\lf(\nu(\phi)\mathbf{D}(\mathbf{u})\rg)\cdot \nabla^{k+1}\mathbf{u} d\mathbf{x}}_{J_1}\\
&+\underline{\sum_{0\leq l\leq k-1}
	C_{k-1}^{l}\int_{\mathbb{T}^3} \nabla^lh_2(\sigma)  \nabla^{k-l}\lf((\nu(\phi)+\lambda(\phi)) {\rm div}\mathbf{u}\rg)\cdot \nabla^{k+1}\mathbf{u} d\mathbf{x}}_{J_2}. \end{aligned}\end{equation}  By using H\"oder inequality, \eqref{h310} and Sobolev embedding, we have
$$\begin{aligned}
J_1&\lesssim\|h_2(\sigma)\|_{L^\infty}\| \nabla\lf(\nu(\phi)\mathbf{D}(\mathbf{u})\rg)\|\|\nabla^{2}\mathbf{u}\|\\
& \lesssim\eta_1\lf(\|\nu(\phi)\|_{L^\infty}\|\nabla^2\mathbf{u}\|+\|\nu'(\phi)\|_{L^\infty}\|\nabla\phi\|\|\nabla\mathbf{u}\|\rg)\|\nabla^{2}\mathbf{u}\|
\\
&\overset{\eqref{hr33}} \lesssim\eta_1\lf(\|\nabla^2\mathbf{u}\|+\|\nabla\mathbf{u}\|\rg)\|\nabla^{2}\mathbf{u}\| \overset{\eqref{hr317}}\lesssim\eta_1\|\nabla^{2}\mathbf{u}\|^2\end{aligned}$$ for $k=1$, and $$\begin{aligned}
J_1&\lesssim \sum_{0\leq l\leq k-1}\|\nabla^lh_2(\sigma) \nabla^{k-l}\lf(\nu(\phi)\mathbf{D}(\mathbf{u})\rg)\|\|\nabla^{k+1}\mathbf{u}\|\\
& \lesssim\|h_2(\sigma)\|_{L^\infty}\| \nabla^{k}\lf(\nu(\phi)\mathbf{D}(\mathbf{u})\rg)\|\|\nabla^{k+1}\mathbf{u}\|\\
&+\sum_{1\leq l\leq k-1}\|\nabla^lh_2(\sigma) \|_{L^3}\| \nabla^{k-l}\lf(\nu(\phi)\mathbf{D}(\mathbf{u})\rg)\|\|\nabla^{k+1}\mathbf{u}\|\\
& \overset{\eqref{hg22}, \eqref{hg21}}\lesssim \eta_1\lf(\|\nu(\phi)\|_{L^\infty}\|\nabla^{k+1}\mathbf{u}\|+\|\nabla^{k}\phi\|\|\nabla\mathbf{u}\|_{L^\infty}\rg) \|\nabla^{k+1}\mathbf{u}\|\\
&+\sum_{1\leq l\leq k-1}\|\nabla^l\sigma \|_{L^3}\lf(\|\nu(\phi)\|_{L^\infty}\|\nabla^{k-l+1}\mathbf{u}\|+\|\nabla^{k-l}\phi\|\|\nabla\mathbf{u}\|_{L^\infty}\rg)\|\nabla^{k+1}\mathbf{u}\|\\
&+\sum_{1\leq l\leq k-1}\|\nabla^l\sigma \|_{L^3}\lf(\|\nu(\phi)\|_{L^\infty}\|\nabla^{k-l+2}\mathbf{u}\|+\|\nabla^{k-l+1}\phi\|\|\nabla\mathbf{u}\|_{L^\infty}\rg)\|\nabla^{k+1}\mathbf{u}\|\\
&\overset{\eqref{h310}, \eqref{hr317}}\lesssim \eta_1\lf(\|\nabla^{k+2}\phi\|^2+\|\nabla^{k+1} \mathbf{u}\|^2\rg)\end{aligned}$$ for $k=2,\cdots,3$. Therefore,  by following the same argument on $J_2$,  we obtain from  \eqref {hr327} that
$$
I_4^3\lesssim \eta_1\lf(\|\nabla^{k+2}\phi\|^2+\|\nabla^{k+1} \mathbf{u}\|^2\rg).
$$
Also, for  $I_{4}^4$, we have
$$
\begin{aligned}
I_4^4&=-\sum_{0\leq l\leq k}
C_{k}^{l}\int_{\mathbb{T}^3} \nabla^{l+1}\phi\nabla^{k-l}\Delta\phi\cdot \nabla^{k}\mathbf{u}d\mathbf{x}\\
&\lesssim \|\nabla\phi\|_{L^\infty}\|\nabla^{k+2}\phi\|\|\nabla^{k}\mathbf{u}\|+\sum_{1\leq l\leq k} \|\nabla^{l+1}\phi\|_{L^6}\|\nabla^{k-l+2}\phi\|_{L^3}\|\nabla^{k}\mathbf{u}\|\\
&\overset{\eqref{h310},\eqref{hr317}}\lesssim \eta_1\lf(\|\nabla^{k+2}\phi\|^2+\|\nabla^{k+1} \mathbf{u}\|^2\rg).\end{aligned}$$
For  $I_{4}^5$, we have
$$
\begin{aligned}
I_4^5&=\sum_{1\leq l\leq k}
C_{k}^{l}\int_{\mathbb{T}^3} \nabla^{l}\nu(\phi)\mathbf{D} (\nabla^{k-l}\mathbf{u}):
\nabla \nabla^{k}\mathbf{u} d\mathbf{x}\\
&\overset{\eqref{hg21}}\lesssim \sum_{1\leq l\leq k} \|\nabla^{l}\phi\|_{L^3}\|\nabla^{k-l+1}\mathbf{u}\|_{L^6}\|\nabla^{k+1}\mathbf{u}\| \overset{\eqref{h310},\eqref{hr317}}\lesssim \eta_1\|\nabla^{k+1} \mathbf{u}\|^2.\end{aligned}$$
By the same lines as in  $I_{4}^5$ and $I_{4}^3$, we have
$$I_4^6\lesssim \eta_1\|\nabla^{k+1} \mathbf{u}\|^2, \ \qquad I_4^7\lesssim \eta_1\big(\|\nabla^{k+2}\phi\|^2+\|\nabla^{k+1} \mathbf{u}\|^2\big).$$
Using \eqref{h327} and the estimates for $I_{4}^j(j=1,\cdots,6)$, we obtain \eqref{hr316} for $k=1,\cdots, 3$ from  \eqref{h324}.
The proof of Lemma \ref{lem32} is completed.\end{proof}

\begin{lemma}\label{lem35}  Under the  assumption \eqref{h310}, it holds   that
 \begin{equation}
	\label{hr326}
	\begin{aligned}
	\f{d}{dt}&\int_{\mathbb{T}^3} \nabla^{k} \mathbf{u}
	\cdot\nabla^{k+1}\sigma d\mathbf{x}+\f{p'(\bar{\rho})}{2\bar{\rho}} \|\nabla^{k+1} \sigma\|^2
	\\
	&\lesssim \eta_1 \lf(	\|\nabla^{k+2} \mathbf{u}\|^2+\|\nabla^{k+3}\phi\|^2\rg)+\|\nabla^{k+1} \mathbf{u}\|^2,\ \ \mathrm{for}\  k=0,\cdots, 2.
	\end{aligned}\end{equation}
\end{lemma}
\begin{proof} Multiplying  $\eqref{h317}_2$
by $\nabla^{k+1}\sigma$, and using $\eqref{h317}_1$, \eqref{h314} and \eqref{h315}, we have
\begin{equation}\label{3413}\f{d}{dt}\int_{\mathbb{T}^3} \nabla^{k} \mathbf{u}\cdot\nabla^{k+1}\sigma d\mathbf{x}+\f{p'(\bar{\rho})}{\bar{\rho}}\|\nabla^{k+1} \sigma\|^2-\bar{\rho}\int_{\mathbb{T}^3}({\rm
	div}\nabla^{k}\mathbf{u})^2d\mathbf{x}=I_{5},
\end{equation}where
\begin{equation}\label{hg327}\begin{aligned}
I_{5}&=\underline{\int_{\mathbb{T}^3} \nabla^{k}{\rm div}(\sigma\mathbf{u}) {\rm
		div}\nabla^{k}\mathbf{u}d\mathbf{x}}_{I_{5}^1}+\underline{\int_{\mathbb{T}^3} \nabla^{k}\lf[-
	(\mathbf{u},\nabla)\mathbf{u}+h_1(\sigma)\nabla\sigma\rg] \cdot \nabla^{k+1}\sigma d\mathbf{x}}_{I_{5}^2}\\
&\underline{-\int_{\mathbb{T}^3} \nabla^{k}\lf[2h_2(\sigma){\rm div}\lf(\nu(\phi)\mathbf{D}(\mathbf{u})\rg)+h_2(\sigma)\nabla\lf((\nu(\phi)+\lambda(\phi)) {\rm div}\mathbf{u}\rg)\rg] \cdot\nabla^{k+1}\sigma d\mathbf{x}}_{I_{5}^3}
\\
&\underline{-\frac{1}{\bar\rho}\int_{\mathbb{T}^3} \nabla^{k}\lf[\nabla\phi\Delta\phi\rg]\cdot\nabla^{k+1}\sigma d\mathbf{x}}_{I_{5}^4}
+\f{2}{\bar{\rho}}\underline{\int_{\mathbb{T}^3} {\rm div} \nabla^{k}\lf[\nu(\phi)\mathbf{D}(\mathbf{u})\rg]\cdot
	\nabla^{k+1}\sigma d\mathbf{x}}_{I_{5}^5}\\
&+\f{1}{\bar{\rho}}\underline{\int_{\mathbb{T}^3} \nabla^{k+1}\lf[(\nu(\phi)+\lambda(\phi)) {\rm div}\mathbf{u}\rg]\cdot
	\nabla^{k+1}\sigma d\mathbf{x}}_{I_{5}^6}\\
&+\underline{\int_{\mathbb{T}^3} \nabla^{k}\lf[h_2(\sigma)\nabla\phi\Delta\phi\rg]\cdot\nabla^{k+1}\sigma d\mathbf{x}}_{I_{5}^7}
.
\end{aligned}
\end{equation}
We   estimate  $I_{5}$.
For the term $I_{5}^1$, we get
$$\begin{aligned}
I_{5}^1&\lesssim\|\nabla^{k+1}(\sigma\mathbf{u})\|\|\nabla^{k+1}\mathbf{u}\|\\
&\overset{\eqref{hg22}}\lesssim\lf(\|\sigma\|_{L^\infty}\|\nabla^{k+1}\mathbf{u}\|+\|\mathbf{u}\|_{L^\infty}\|\nabla^{k+1}\sigma\|\rg) \|\nabla^{k+1}\mathbf{u}\|\\
&\overset{\eqref{h310}, \eqref{hr317}}\lesssim\eta_1\lf(\|\nabla^{k+2}\mathbf{u}\|^2+\|\nabla^{k+1}\sigma\|^2\rg).
\end{aligned}$$
 Next, for   $I_{5}^2$, we have $$\begin{aligned}
I_{5}^2&\lesssim\|\nabla^{k}\lf[-
(\mathbf{u},\nabla)\mathbf{u}+h_1(\sigma)\nabla\sigma\rg]\|\|\nabla^{k+1}\sigma\|\\
&\overset{\eqref{hg22}}\lesssim \lf(\|\mathbf{u}\|_{L^\infty}\|\nabla^{k+1}\mathbf{u}\|+\|\nabla^{k}\mathbf{u}\|\|\nabla\mathbf{u}\|_{L^\infty}\rg)\|\nabla^{k+1}\sigma\|\\
&+\lf(\|h_1(\sigma)\|_{L^\infty}\|\nabla^{k+1}\sigma\|+\|\nabla^{k}h_1(\sigma)\|\|\nabla\sigma\|_{L^\infty}\rg)\|\nabla^{k+1}\sigma\|\\
&\overset{\eqref{h310}, \eqref{hg21}}\lesssim \eta_1\lf(\|\nabla^{k+1}\mathbf{u}\|+\|\nabla^{k}\mathbf{u}\|+\|\nabla^{k+1}\sigma\|+\|\nabla^{k}\sigma\|\rg)\|\nabla^{k+1}\sigma\| \\
&\overset{\eqref{hr317}}\lesssim \eta_1\lf(\|\nabla^{k+1}\sigma\|^2+\|\nabla^{k+2} \mathbf{u}\|^2\rg).
\end{aligned}
$$
By the same lines as above,  noticing that
$$\begin{aligned}
I_{5}^3&=-\sum_{0\leq l\leq k}
C_{k}^{l}\int_{\mathbb{T}^3} \nabla^lh_2(\sigma) {\rm div} \nabla^{k-l-1}\lf(\nu(\phi)\mathbf{D}(\mathbf{u})\rg)\cdot \nabla^{k+1}\sigma d\mathbf{x}\\
&-\sum_{0\leq l\leq k}
C_{k-1}^{l}\int_{\mathbb{T}^3} \nabla^lh_2(\sigma)  \nabla^{k-l}\lf((\nu(\phi)+\lambda(\phi)) {\rm div}\mathbf{u}\rg)\cdot \nabla^{k+1}\sigma d\mathbf{x}, \end{aligned}$$
we have
$$
I_5^3\lesssim \eta_1\lf(\|\nabla^{k+1}\sigma\|^2+\|\nabla^{k+2} \mathbf{u}\|^2+\|\nabla^{k+3}\phi\|^2\rg).
$$

Also, for $I_{5}^4$, we have
$$
\begin{aligned}
I_5^4&=-\sum_{0\leq l\leq k}
C_{k}^{l}\int_{\mathbb{T}^3} \nabla^{l+1}\phi\nabla^{k-l}\Delta\phi\cdot \nabla^{k+1}\sigma d\mathbf{x}\\
&\lesssim \|\nabla\phi\|_{L^\infty}\|\nabla^{k+2}\phi\|\|\nabla^{k+1}\sigma \|+\sum_{1\leq l\leq k} \|\nabla^{l+1}\phi\|_{L^6}\|\nabla^{k-l+2}\phi\|_{L^3}\|\nabla^{k+1}\sigma \|\\
&\overset{\eqref{h310},\eqref{hr317}}\lesssim \eta_1\lf(\|\nabla^{k+3}\phi\|^2+\|\nabla^{k+1}\sigma \|^2\rg).\end{aligned}$$
For  $I_{5}^5$, we have
 $$ \begin{aligned}
I_5^5&=\int_{\mathbb{T}^3} {\rm div} \nabla^{k}\lf[\nu(\phi)\mathbf{D}(\mathbf{u})\rg]\cdot
\nabla^{k+1}\sigma d\mathbf{x}\\
&=\sum_{0\leq l\leq k}
C_{k+1}^{l}\int_{\mathbb{T}^3} \lf(\nabla^{l}\nu(\phi) {\rm div}\mathbf{D} (\nabla^{k-l}\mathbf{u})+\nabla^{l+1}\nu(\phi)\mathbf{D} (\nabla^{k-l}\mathbf{u})\rg) \cdot
\nabla^{k+1}\sigma d\mathbf{x}\\
&\overset{\eqref{hg21}}\lesssim \sum_{1\leq l\leq k} \lf(\|\nabla^{l}\phi\|_{L^3}\|\nabla^{k-l+2}\mathbf{u}\|_{L^6}+\|\nabla^{l+1}\phi\|_{L^3}\|\nabla^{k-l+1}\mathbf{u}\|_{L^6}\rg)\|\nabla^{k+1}\sigma\|\\
&+\lf(\|\nu(\phi)\|_{L^\infty}\|\nabla^{k+2}\mathbf{u}\|+\|\nabla\nu(\phi)\|_{L^3}\|\nabla^{k-l+1}\mathbf{u}\|_{L^6}\rg)\|\nabla^{k+1}\sigma\|\end{aligned}$$
 $$ \overset{\eqref{h310},\eqref{hr317}}\lesssim \eta_1\lf(\|\nabla^{k+1}\sigma\|^2+\|\nabla^{k+2} \mathbf{u}\|^2\rg).$$
By the same lines as in   $I_{5}^5$ and $I_{5}^4$, we have
$$I_5^6\lesssim \eta_1\lf(\|\nabla^{k+1}\sigma\|^2+\|\nabla^{k+2} \mathbf{u}\|^2\rg),\qquad I_5^7\lesssim \eta_1\lf(\|\nabla^{k+3}\phi\|^2+\|\nabla^{k+1}\sigma \|^2\rg).$$
Using  \eqref{hg327} and the estimates for $I_{5}^j(j=1,\cdots,7)$, we obtain  \eqref{hr326} from   \eqref{3413}.
The proof of Lemma \ref{lem35} is completed.\end{proof}

\vspace{3ex}To get the decay of the solution, we need the following evolution of the negative Sobolev norms of the solution to the system \eqref{h21}.
\begin{lemma}\label{lem41} Under the  assumption \eqref{h310}, it holds   that
	\begin{equation}
	\label{hr41}\begin{aligned}
&\f{d}{dt}\lf(\f{p'(\bar{\rho})}{\bar{\rho}^2}\|\Lambda^{-s} \sigma\|^2+\|\Lambda^{-s} \mathbf{u}\|^2+\|\Lambda^{-s} \nabla\phi\|^2\rg)\\
&
\lesssim \lf(\|\nabla(\sigma,  \mathbf{u})\|_{H^2}^2+\|\nabla\phi\|_{H^3}^2\rg)\lf(\|\Lambda^{-s} \sigma\|+\|\Lambda^{-s} \mathbf{u}\|+\|\Lambda^{-s} \nabla\phi\|\rg),\ \ 	\mathrm{for} s\in (0, \f{3}{2}).\end{aligned}
		\end{equation}
\end{lemma}
\begin{proof}
Applying $\Lambda^{-s}$ to $\eqref{h313}_1$,  $\eqref{h313}_2$ and \eqref{h3280}, multiplying the resulting identities by $\f{P'(\bar{\rho})}{\bar{\rho}^2}\Lambda^{-s}\sigma$, $\Lambda^{-s}\mathbf{u}$  and $-\Lambda^{-s}\Delta\phi$, respectively, summing up and then integrating by parts, we deduce that
	\begin{equation}
\label{hr44}
\begin{aligned}
&\f{1}{2}\f{d}{dt}\lf(\f{p'(\bar{\rho})}{\bar{\rho}^2}\|\Lambda^{-s} \sigma\|^2+\|\Lambda^{-s} \mathbf{u}\|^2+\|\Lambda^{-s} \nabla\phi\|^2\rg)+\f{\nu'}{\bar{\rho}}\|\Lambda^{-s} \nabla\mathbf{u}\|^2\\
&+\f{\nu'+\lambda'}{\bar{\rho}}\|\Lambda^{-s} {\rm div}\mathbf{u}\|^2+\f{1}{\bar{\rho}^2}\|\Lambda^{-s}\nabla \Delta\phi\|^2\\
&=-\underline{\f{2}{\bar{\rho}}\int_{\mathbb{T}^3} \Lambda^{-s}{\rm div}\lf[(\nu(\phi)-\nu')\mathbf{D}(\mathbf{u})\rg]\Lambda^{-s}\mathbf{u}d\mathbf{x}}_{I_6}\\&
-\underline{\f{1}{\bar{\rho}}\int_{\mathbb{T}^3} \Lambda^{-s}\nabla\lf[((\nu(\phi)-\nu')+(\lambda(\phi)-\lambda') {\rm div}\mathbf{u}\rg]\Lambda^{-s}\mathbf{u}d\mathbf{x}}_{I_7}\\
&
-\f{p'(\bar{\rho})}{\bar{\rho}^2}\underline{\int_{\mathbb{T}^3} \Lambda^{-s}\lf(\sigma{\rm div}\mathbf{u}\!+\!\nabla \sigma\!\cdot\!\mathbf{u}\rg)\Lambda^{-s}\sigma d\mathbf{x}}_{I_8}\!-\!\underline{\int_{\mathbb{T}^3} \Lambda^{-s}\lf[(\mathbf{u}\cdot\nabla)\mathbf{u}\!-\!h_1(\sigma)\nabla\sigma\rg]\cdot \Lambda^{-s}\mathbf{u}  d\mathbf{x}}_{I_9}\\
&-\underline{\int_{\mathbb{T}^3} \Lambda^{-s}\lf[2h_2(\sigma){\rm div}\lf[\nu(\phi)\mathbf{D}(\mathbf{u})\rg]+h_2(\sigma)\nabla\lf[(\nu(\phi)+\lambda(\phi)) {\rm div}\mathbf{u}\rg]\rg]\cdot \Lambda^{-s}\mathbf{u} d\mathbf{x}}_{I_{10}}\\
&-\underline{\frac{1}{\bar\rho}\int_{\mathbb{T}^3} \Lambda^{-s}\lf(\nabla\phi\Delta\phi\rg)\cdot\Lambda^{-s}\mathbf{u} d\mathbf{x}}_{I_{11}}+\underline{\int_{\mathbb{T}^3} \Lambda^{-s}\nabla\lf( \f{3\phi^2-1}{\rho}\Delta\phi+\f{6\phi}{\rho}(\nabla\phi)^2\rg)\Lambda^{-s}\nabla \phi d\mathbf{x}}_{I_{12}}\\
&+\underline{\int_{\mathbb{T}^3} \Lambda^{-s}\nabla\lf(-\mathbf{u}\cdot\nabla\phi+h_3(\sigma)\Delta^2\phi+\f{2}{\rho^3}\Delta\nabla\phi\cdot\nabla\sigma+\f{\Delta\phi}{\rho}{\rm div}(\f{\nabla\sigma}{\rho^2})\rg)\Lambda^{-s}\nabla \phi d\mathbf{x}}_{I_{13}}\\
&+\underline{\int_{\mathbb{T}^3} \Lambda^{-s}\lf(h_2(\sigma)\nabla\phi\Delta\phi\rg)\cdot\Lambda^{-s}\mathbf{u} d\mathbf{x}}_{I_{14}},
\end{aligned}
\end{equation}where
$$\nu'=\nu\lf(\bbint_{\mathbb{T}^3}\phi d\mathbf{x}\rg),\,\,\, \lambda'=\lambda\lf(\bbint_{\mathbb{T}^3} \phi d\mathbf{x}\rg)\,\,\, \text{and}\,\,\,h_3(\sigma)=\f{1}{\bar{\rho}^2}-\f{1}{\rho^2}.$$
In order to estimate the nonlinear terms in the right-hand side  of \eqref{hr44}, we shall use the estimate \eqref{hg25}. If $s\in (0, \f{3}{2})$, then $\f{1}{2}+\f{s}{3}<1$ and  we have
$$\begin{aligned}
I_{6}&\lesssim \|\Lambda^{-s}{\rm div}\lf[(\nu(\phi)-\nu')\mathbf{D}(\mathbf{u})\rg]\|\|\Lambda^{-s}\mathbf{u}\|
\\&\overset{\eqref{hg25}}\lesssim \|{\rm div}\lf[(\nu(\phi)-\nu')\mathbf{D}(\mathbf{u})\rg]\|_{L^{\f{1}{\f{1}{2}+\f{s}{3}}}}\|\Lambda^{-s}\mathbf{u}\|\\
&\lesssim \|{\rm div}[\nu(\phi)-\nu']\|_{L^\f{3}{s}}\|\nabla\mathbf{u}\|\|\Lambda^{-s}\mathbf{u}\|\\
&\lesssim \lf(\|\nabla \nu(\phi)\|_{L^\f{3}{s}}\|\nabla\mathbf{u}\|+{\Big \|}\phi-\bbint_{\mathbb{T}^3} \phi d\mathbf{x}{\Big \|}_{L^\f{3}{s}}\|\nabla^2\mathbf{u}\|\rg)\|\Lambda^{-s}\mathbf{u}\|
\\&\overset{\eqref{hr33}, \eqref{h45}}\lesssim \|\nabla\phi\|_{L^\f{3}{s}}\|\nabla\mathbf{u}\|\|\Lambda^{-s}\mathbf{u}\|
\overset{\eqref{hh45}}\lesssim \|\nabla\phi\|_{H^2} \|\nabla\mathbf{u}\|\|\Lambda^{-s}\mathbf{u}\|
\\&\lesssim\lf(\|\nabla\phi\|_{H^2}^2+\|\nabla\mathbf{u}\|^2\rg)\|\Lambda^{-s}\mathbf{u}\|,
\end{aligned}$$
where we used
 the Sobolev inequality
\begin{equation}\label{hh45}\|w\|_{L^p}\lesssim \|w\|_{H^2}\quad  \text{for all}\,\,\, w\in H^{2}(\mathbb{T}^3),\,\, 1\leq p\leq\infty. \end{equation}
Similarly, we  can prove that
$$I_7\lesssim\lf(\|\nabla\phi\|_{H^2}^2+\|\nabla\mathbf{u}\|^2\rg)\|\Lambda^{-s}\mathbf{u}\|.$$
To estimate $I_8$ and $I_9$, we rewrite them respectively as
$$\begin{aligned}
I_8&=\int_{\mathbb{T}^3} \Lambda^{-s}\lf[\sigma{\rm div}\mathbf{u}+\nabla \sigma\cdot(\mathbf{u}-\mathbf{u}')\rg]\Lambda^{-s}\sigma d\mathbf{x}\\
\end{aligned}$$ and $$\begin{aligned}
I_9&=\int_{\mathbb{T}^3} \Lambda^{-s}\lf[(\mathbf{u}-\mathbf{u}',\nabla)\mathbf{u}-h_1(\sigma)\nabla\sigma\rg]\cdot \Lambda^{-s}\mathbf{u}  d\mathbf{x},\\
\end{aligned}$$
where
$\displaystyle\mathbf{u}'=\bbint_{\mathbb{T}^3} \mathbf{u} d\mathbf{x}$ and we used
$$ \int_{\mathbb{T}^3} \mathbf{u}'\cdot\nabla \Lambda^{-s}\sigma\Lambda^{-s}\sigma d\mathbf{x}=0, \quad \int_{\mathbb{T}^3} (\mathbf{u}', \nabla) \Lambda^{-s}\mathbf{u}\cdot\Lambda^{-s}\mathbf{u} d\mathbf{x}=0.$$
Then, we have
$$\begin{aligned}
I_{8}&\lesssim \|\Lambda^{-s}\lf[\sigma{\rm div}\mathbf{u}+\nabla \sigma\cdot(\mathbf{u}-\mathbf{u}')\rg]\|\|\Lambda^{-s}\sigma\|
\\&\overset{\eqref{hg25}}\lesssim \|\sigma{\rm div}\mathbf{u}+\nabla \sigma\cdot(\mathbf{u}-\mathbf{u}')\|_{L^{\f{1}{\f{1}{2}+\f{s}{3}}}}\|\Lambda^{-s}\sigma\|\\
&\lesssim\lf(\|\sigma\|_{L^\f{3}{s}}\|\nabla\mathbf{u}\|+\|\nabla\sigma\|_{L^\f{3}{s}}\|\mathbf{u}-\mathbf{u}'\|\rg) \|\Lambda^{-s}\sigma\|
\\&\overset{\eqref{hh45}, \eqref{h45}}\lesssim \|\sigma\|_{H^3} \|\nabla\mathbf{u}\|\|\Lambda^{-s}\sigma\|
\overset{\eqref{h45}}\lesssim\lf(\|\nabla\sigma\|_{H^2}^2+\|\nabla\mathbf{u}\|^2\rg)\|\Lambda^{-s}\sigma\|,
\end{aligned}$$ and
$$\begin{aligned}
I_{9}&\lesssim \|\Lambda^{-s}\lf[(\mathbf{u}-\mathbf{u}',\nabla)\mathbf{u}-h_1(\sigma)\nabla\sigma\rg] \|\|\Lambda^{-s}\mathbf{u}\|
\\&\overset{\eqref{hg25}}\lesssim \|(\mathbf{u}-\mathbf{u}',\nabla)\mathbf{u}-h_1(\sigma)\nabla\sigma\|_{L^{\f{1}{\f{1}{2}+\f{s}{3}}}}\|\Lambda^{-s}\mathbf{u}\|\\
&\lesssim \lf(\|\mathbf{u}-\mathbf{u}'\|_{L^\f{3}{s}}\|\nabla\mathbf{u}\|+\|\sigma\|_{L^\f{3}{s}}\|\nabla\sigma\|\rg)\|\Lambda^{-s}\mathbf{u}\|
\\&\overset{\eqref{h45}, \eqref{hh45}}\lesssim \lf(\|\nabla\mathbf{u}\|_{H^2}^2+\|\sigma\|_{H^2}^2\rg) \|\Lambda^{-s}\mathbf{u}\|\\&
\overset{\eqref{h45}}\lesssim \lf(\|\nabla\mathbf{u}\|_{H^2}^2+\|\nabla\sigma\|_{H^1}^2\rg)\|\Lambda^{-s}\mathbf{u}\|.
\end{aligned}$$
Also,  we have
$$\begin{aligned}
I_{10}&\lesssim \lf(\|\Lambda^{-s}[h_2(\sigma){\rm div}\lf[\nu(\phi)\mathbf{D}(\mathbf{u})\rg]]\|+\|\Lambda^{-s}[h_2(\sigma)\nabla\lf[(\nu(\phi)+\lambda(\phi)) {\rm div}\mathbf{u}\rg]]\|\rg)\|\Lambda^{-s}\mathbf{u}\|
\\&\overset{\eqref{hg25}}\lesssim \lf(\|h_2(\sigma){\rm div}\lf[\nu(\phi)\mathbf{D}(\mathbf{u})\rg]\|_{L^{\f{1}{\f{1}{2}+\f{s}{3}}}}+\|h_2(\sigma)\nabla\lf[(\nu(\phi)+\lambda(\phi)) {\rm div}\mathbf{u}\rg]\|_{L^{\f{1}{\f{1}{2}+\f{s}{3}}}}\rg)\|\Lambda^{-s}\mathbf{u}\|\\
&\lesssim \|h_2(\sigma)\|_{L^\f{3}{s}}\lf(\|{\rm div}\lf[\nu(\phi)\mathbf{D}(\mathbf{u})\rg]\|+\|\nabla\lf[(\nu(\phi)+\lambda(\phi)) {\rm div}\mathbf{u}\rg]\|\rg)\|\Lambda^{-s}\mathbf{u}\|
\\&\overset{\eqref{hg21}, \eqref{h310}}\lesssim \|\sigma\|_{L^\f{3}{s}}\|\nabla\mathbf{u}\|_{H^1}\|\Lambda^{-s}\mathbf{u}\|\\&
\overset{\eqref{hh45}, \eqref{h45}}\lesssim\lf(\|\nabla\sigma\|_{H^1}^2+\|\nabla\mathbf{u}\|_{H^1}^2\rg)\|\Lambda^{-s}\mathbf{u}\|
\end{aligned}$$  and
$$\begin{aligned}
I_{11}&\lesssim \|\Lambda^{-s}\lf(\nabla\phi\Delta\phi\rg)\| \|\Lambda^{-s}\mathbf{u}\|
\\&\overset{\eqref{hg25}}\lesssim \|\nabla\phi\Delta\phi\|_{L^{\f{1}{\f{1}{2}+\f{s}{3}}}}\|\Lambda^{-s}\mathbf{u}\| \lesssim \|\nabla\phi\|_{L^\f{3}{s}}\|\nabla^2\phi\|\|\Lambda^{-s}\mathbf{u}\|
\\&
\overset{\eqref{hh45}}
\lesssim\|\nabla\phi\|_{H^2}^2\|\Lambda^{-s}\mathbf{u}\|.
\end{aligned}$$
The estimate on $I_{12}$ is more
 subtle. To this end, we rewrite it as
\begin{equation}\label{hr45}
\begin{aligned}
I_{12}=&\underline{\int_{\mathbb{T}^3} \Lambda^{-s}\nabla\lf[(3\phi^2-1)\lf( \f{1}{\rho}-\f{1}{\bar{\rho}}\rg)\Delta\phi\rg]\Lambda^{-s}\nabla \phi d\mathbf{x}}_{I_{12}^1}\\
&+\f{3}{\bar{\rho}}\underline{\int_{\mathbb{T}^3} \Lambda^{-s}\nabla\lf( (\phi^2-\overline{\phi^{2}})\Delta\phi\rg)\Lambda^{-s}\nabla \phi d\mathbf{x}}_{I_{12}^2}
	\\&-\f{3\overline{\phi^{2}}-1}{\bar{\rho}}\underline{\int_{\mathbb{T}^3} |\Lambda^{-s}\nabla \phi|^2 d\mathbf{x}}_{I_{12}^3}+6\underline{\int_{\mathbb{T}^3} \Lambda^{-s}\nabla\lf(\f{\phi}{\rho}(\nabla\phi)^2\rg)\Lambda^{-s}\nabla \phi d\mathbf{x}}_{I_{12}^4},
\end{aligned} \end{equation}  where $\displaystyle\overline{\phi^{2}}=\bbint_{\mathbb{T}^3} \phi^2 d\mathbf{x}.$
 Then, we have
$$\begin{aligned}
I_{12}^1&\lesssim {\Big\|}\Lambda^{-s}\nabla\lf[(3\phi^2-1)\lf( \f{1}{\rho}-\f{1}{\bar{\rho}}\rg)\Delta\phi\rg]{\Big\|} \|\Lambda^{-s}\nabla\phi\|\\
& \overset{\eqref{hg25}}\lesssim {\Big\|}\nabla\lf[(3\phi^2-1)\lf( \f{1}{\rho}-\f{1}{\bar{\rho}}\rg)\Delta\phi\rg]{\Big\|}_{L^{\f{1}{\f{1}{2}+\f{s}{3}}}}\|\Lambda^{-s}\nabla\phi\|\\
& \lesssim \lf(\|\sigma\|_{L^\f{3}{s}}\|\nabla^3\phi\|+\|\nabla\phi\|_{L^\f{3}{s}}\|\nabla^2\phi\|+\|\nabla\sigma\|_{L^\f{3}{s}}\|\nabla^2\phi\|\rg) \|\Lambda^{-s}\nabla\phi\|
\\
&\overset{\eqref{hh45}, \eqref{h45}}
\lesssim\lf(\|\nabla\sigma\|_{H^2}^2+\|\nabla\phi\|_{H^2}^2\rg)\|\Lambda^{-s}\nabla\phi\|,
\end{aligned}$$
and
$$\begin{aligned}
I_{12}^2&\lesssim {\Big\|}\Lambda^{-s}\nabla\lf( (\phi^2-\overline{\phi^{2}})\Delta\phi\rg){\Big\|} \|\Lambda^{-s}\nabla\phi\|\\&
 \overset{\eqref{hg25}}\lesssim {\Big\|}\nabla\lf((\phi^2-\overline{\phi^{2}})\Delta\phi\rg){\Big\|}_{L^{\f{1}{\f{1}{2}+\f{s}{3}}}}\|\Lambda^{-s}\nabla\phi\|\\
& \lesssim \lf(\|\phi^2-\overline{\phi^{2}}\|_{L^\f{3}{s}}\|\nabla^3\phi\|+\|\nabla\phi\|_{L^\f{3}{s}}\|\nabla^2\phi\|\rg) \|\Lambda^{-s}\nabla\phi\|
\\
&\overset{\eqref{hh45}, \eqref{h45}}
\lesssim \|\nabla\phi\|_{H^2}^2\|\Lambda^{-s}\nabla\phi\|,
\end{aligned}$$
moreover,
\begin{eqnarray*}
& \displaystyle I_{12}^3+\f{2}{\bar{\rho}}\int_{\mathbb{T}^3} |\Lambda^{-s}\nabla \phi|^2 d\mathbf{x} \\
& \displaystyle \leq \f{3}{\bar{\rho}}(1-\overline{\phi^{2}})\int_{\mathbb{T}^3} |\Lambda^{-s}\nabla \phi|^2 d\mathbf{x} \\
& \displaystyle \lesssim \eta_1\int_{\mathbb{T}^3} |\Lambda^{-s}\nabla \phi|^2 d\mathbf{x},
\end{eqnarray*}
 where we used
\begin{equation}\label{hr411}1-\overline{\phi^{2}}=\bbint_{\mathbb{T}^3}(1-\phi^2)d\mathbf{x}\lesssim \|1-\phi^2\|\lesssim \eta_1,\end{equation}
and
$$\begin{aligned}
I_{12}^4&\lesssim {\Big\|}\Lambda^{-s}\nabla\lf(\f{\phi}{\rho}(\nabla\phi)^2\rg){\Big\|} \|\Lambda^{-s}\nabla\phi\|
\overset{\eqref{hg25}}\lesssim {\Big\|}\nabla\lf(\f{\phi}{\rho}(\nabla\phi)^2\rg){\Big\|}_{L^{\f{1}{\f{1}{2}+\f{s}{3}}}}\|\Lambda^{-s}\nabla\phi\|\\
& \lesssim \|\nabla\phi\|_{L^\f{3}{s}} \lf(\|\nabla\phi\|+\|\nabla^2\phi\|\rg) \|\Lambda^{-s}\nabla\phi\|
\\
&\overset{\eqref{hh45}, \eqref{h45}}
\lesssim \|\nabla\phi\|_{H^2}^2\|\Lambda^{-s}\nabla\phi\|.
\end{aligned}$$  Therefore, we obtain from  \eqref{hr45} that
$$I_{12}\lesssim\lf(\|\nabla\sigma\|_{H^2}^2+\|\nabla\phi\|_{H^2}^2\rg)\|\Lambda^{-s}\nabla\phi\|.$$
Last, for $I_{13}$, noticing that
$\int_{\mathbb{T}^3} \mathbf{u}'\cdot \Lambda^{-s}\nabla\nabla\phi \Lambda^{-s}\nabla \phi d\mathbf{x}=0,$
 we have
$$\begin{aligned}
I_{13}&\lesssim \|\Lambda^{-s}\nabla\lf[(\mathbf{u}'-\mathbf{u})\cdot\nabla\phi+h_3(\sigma)\Delta^2\phi\rg]\| \|\Lambda^{-s}\nabla\phi\|\\
&+{\Big\|}\Lambda^{-s}\nabla\lf(\f{2}{\rho^3}\Delta\nabla\phi\cdot\nabla\sigma+\f{\Delta\phi}{\rho}{\rm div}(\f{\nabla\sigma}{\rho^2})\rg){\Big\|} \|\Lambda^{-s}\nabla\phi\|\\
&
\overset{\eqref{hg25}}\lesssim\|\nabla\lf[(\mathbf{u}'-\mathbf{u})\cdot\nabla\phi+h_3(\sigma)\Delta^2\phi\rg]\|_{L^{\f{1}{\f{1}{2}+\f{s}{3}}}} \|\Lambda^{-s}\nabla\phi\|\\
&+{\Big\|}\nabla\lf(\f{2}{\rho^3}\Delta\nabla\phi\cdot\nabla\sigma+\f{\Delta\phi}{\rho}{\rm div}(\f{\nabla\sigma}{\rho^2})\rg){\Big\|}_{L^{\f{1}{\f{1}{2}+\f{s}{3}}}} \|\Lambda^{-s}\nabla\phi\|\\
& \lesssim  \lf[(\|\mathbf{u}-\mathbf{u}'\|_{L^\f{3}{s}}+\|\nabla\sigma\|_{L^\f{3}{s}})\|\nabla^2\phi\| +\|\nabla\mathbf{u}\|_{L^\f{3}{s}}\|\nabla\phi\|+\|\sigma\|_{L^\f{3}{s}}\|\nabla^3\phi\|\rg]\|\Lambda^{-s}\nabla\phi\|\\
&+\lf[\|\nabla\sigma\|_{L^\f{3}{s}}\|\Delta\nabla\phi\|_{H^1}+\|\nabla\phi\|_{W^{1,\f{3}{s}}}\|\Delta\sigma\|_{H^1}\rg]\|\Lambda^{-s}\nabla\phi\|
\\
&\overset{\eqref{hh45}, \eqref{h45}}
\lesssim \lf(\|\nabla\mathbf{u}\|_{H^2}^2+\|\nabla\sigma\|_{H^2}^2+\|\nabla\phi\|_{H^3}^2\rg) \|\Lambda^{-s}\nabla\phi\|.
\end{aligned}$$
By the same lines as in $I_{11}$, we obtain
$$I_{14}\lesssim\|\nabla\phi\|_{H^2}^2\|\Lambda^{-s}\mathbf{u}\|.$$
Substituting the estimates on $I_{i} (i=6, \cdots, 14)$ into \eqref{hr44} yields \eqref{hr41}. The proof of Lemma \ref{lem41} is completed.\end{proof}

\begin{lemma}\label{lem42} Under the  assumption \eqref{h310}, it holds   that
	\begin{equation}
	\label{hr49}
\f{d}{dt}\|\Lambda^{-s} \lf(\phi^2-1\rg)\|^2
	\lesssim \lf(\|\nabla(\sigma,  \mathbf{u})\|_{H^2}^2+\|\nabla\phi\|_{H^3}^2\rg)\|\Lambda^{-s} \lf(\phi^2-1\rg)\|, \ \ \mathrm{for} \ s\in (0, \f{3}{2}).
	\end{equation}
\end{lemma}
\begin{proof} We rewrite \eqref{h3280} as
\begin{equation}\label{hr330}
\begin{aligned}
&\lf(\phi^2-1\rg)_t+ \mathbf{u}\cdot\nabla\lf(\phi^2-1\rg)-\f{2}{\bar{\rho}}\Delta(\phi^2-1)+\f{1}{\bar{\rho}^2}\Delta^2(\phi^2-1)\\
&=\f{12\phi^3}{\rho}|\nabla\phi|^2+\f{6\phi}{\rho}\lf(\phi^2-1\rg)\Delta\phi+4\phi\lf(\f{1}{\rho}-\f{1}{\bar{\rho}}\rg)\Delta\phi-\f{2}{\bar{\rho}}|\nabla\phi|^2\\
&+\f{2}{\bar{\rho}^2}\lf[|\nabla\phi|^2+2\nabla\phi\cdot\nabla\Delta\phi+\Delta(|\nabla\phi|^2)\rg]\\
&+2\phi\lf(\f{1}{\rho}-\f{1}{\bar{\rho}}\rg)\Delta\lf(\f{\Delta\phi}{\rho}\rg)+\f{2\phi}{\bar{\rho}}\Delta\lf[\lf(\f{1}{\rho}-\f{1}{\bar{\rho}}\rg)\Delta\phi\rg].
\end{aligned} \end{equation}
Then, applying $\Lambda^{-s}$ to \eqref{hr330} and multiplying the resulting identity by  $\Lambda^{-s}\lf(\phi^2-1\rg)$,  we deduce that
\begin{equation}
\label{hh44}
\begin{aligned}
&\f{1}{2}\f{d}{dt}\|\Lambda^{-s} \lf(\phi^2-1\rg)\|^2+\f{2}{\bar{\rho}}\|\Lambda^{-s} \nabla\lf(\phi^2-1\rg)\|^2+\f{1}{\bar{\rho}^2}\|\Lambda^{-s} \Delta\lf(\phi^2-1\rg)\|^2\\&
=
\underline{\int_{\mathbb{T}^3} \Lambda^{-s}\lf[\lf(\f{12\phi^3}{\rho}-\f{2}{\bar{\rho}}\rg)|\nabla\phi|^2+4\phi h_{2}(\sigma)\Delta\phi\rg]\Lambda^{-s}\lf(\phi^2-1\rg) d\mathbf{x}}_{I_{15}}\\&+\f{2}{\bar{\rho}^2}\underline{\int_{\mathbb{T}^3} \Lambda^{-s}\lf[|\nabla\phi|^2+2\nabla\phi\cdot\nabla\Delta\phi+\Delta(|\nabla\phi|^2)\rg]\Lambda^{-s}\lf(\phi^2-1\rg) d\mathbf{x}}_{I_{16}}
\\
&-2\underline{\int_{\mathbb{T}^3} \Lambda^{-s}\lf[\phi h_{2}(\sigma)\Delta\lf(\f{\Delta\phi}{\rho}\rg)+\f{\phi}{\bar{\rho}}\Delta\lf(h_{2}(\sigma)\Delta\phi\rg)\rg]\Lambda^{-s}\lf(\phi^2-1\rg) d\mathbf{x}}_{I_{17}}\\
&-\underline{\int_{\mathbb{T}^3} \Lambda^{-s}\lf[\mathbf{u}\cdot\nabla\lf(\phi^2-1\rg)\rg]\Lambda^{-s}\lf(\phi^2-1\rg) d\mathbf{x}}_{I_{18}}\\
&+6\underline{\int_{\mathbb{T}^3} \Lambda^{-s}\lf[\f{\phi}{\rho}\lf(\phi^2-1\rg)\Delta\phi\rg]\Lambda^{-s}\lf(\phi^2-1\rg) d\mathbf{x}}_{I_{19}},
\end{aligned}
\end{equation}
where
$h_2(\sigma)=\f{1}{\bar{\rho}}-\f{1}{\rho}$ (see \eqref{h315}).
In order to estimate the nonlinear terms in the right-hand side  of \eqref{hh44}, we shall use the estimate \eqref{hg25}. If $s\in (0, \f{3}{2})$, then $\f{1}{2}+\f{s}{3}<1$ and  we have
$$\begin{aligned}
I_{15}&\lesssim {\Big\|}\Lambda^{-s}\lf[\lf(\f{12\phi^3}{\rho}-\f{2}{\bar{\rho}}\rg)|\nabla\phi|^2+4\phi h_{2}(\sigma)\Delta\phi\rg]{\Big\|} \|\Lambda^{-s}\lf(\phi^2-1\rg) \|\\
&
\overset{\eqref{hg25}, \eqref{hr33}, \eqref{hg313}}\lesssim\lf(\|\nabla\phi\|_{L^{\f{3}{s}}}\|\nabla\phi\|+\|\Delta\phi\|\|\sigma\|_{L^{\f{3}{s}}}\rg) \|\Lambda^{-s}\lf(\phi^2-1\rg) \|
\\
&\overset{\eqref{hh45}, \eqref{h45}}
\lesssim \lf(\|\nabla\phi\|_{H^3}^2+\|\nabla\sigma\|_{H^2}^2\rg) \|\Lambda^{-s}\lf(\phi^2-1\rg)\|,
\end{aligned}$$
$$\begin{aligned}
I_{16}&\lesssim \|\Lambda^{-s}\lf[|\nabla\phi|^2+2\nabla\phi\cdot\nabla\Delta\phi+\Delta(|\nabla\phi|^2)\rg]\| \|\Lambda^{-s}\lf(\phi^2-1\rg) \|\\
&
\overset{\eqref{hg25}}\lesssim \lf[\|\nabla\phi\|_{L^{\f{3}{s}}}\lf(\|\nabla\phi\|+\|\nabla\Delta\phi\|\rg)+\|\Delta\phi\|_{L^{\f{3}{s}}}\|\Delta\phi\|\rg] \|\Lambda^{-s}\lf(\phi^2-1\rg) \|
\\
&\overset{ \eqref{h45}}
\lesssim \|\nabla\phi\|_{H^3}^2 \|\Lambda^{-s}\lf(\phi^2-1\rg)\|,
\end{aligned}$$
and
$$\begin{aligned}
I_{17}&\lesssim {\Big\|}\Lambda^{-s}\lf[\phi h_{2}(\sigma)\Delta\lf(\f{\Delta\phi}{\rho}\rg)+\f{\phi}{\bar{\rho}}\Delta\lf(h_{2}(\sigma)\Delta\phi\rg)\rg]{\Big\|} \|\Lambda^{-s}\lf(\phi^2-1\rg) \|\\
&
\overset{\eqref{hg25}, \eqref{hr33}, \eqref{hg313}}\lesssim\|\sigma\|_{L^{\f{3}{s}}}\lf({\Big\|}\Delta\lf(\f{\Delta\phi}{\rho}\rg){\Big\|}+\|\Delta^2\phi\|\rg) \|\Lambda^{-s}\lf(\phi^2-1\rg) \|\\
&+\lf(\|\nabla\sigma\|_{L^{\f{3}{s}}}
\|\nabla\Delta\phi\|+\|\Delta\sigma\|\|\Delta\phi\|_{L^{\f{3}{s}}}
\rg) \|\Lambda^{-s}\lf(\phi^2-1\rg) \|
\\
&\overset{\eqref{hh45}, \eqref{h45}}
\lesssim \lf(\|\nabla\phi\|_{H^3}^2+\|\nabla\sigma\|_{H^2}^2\rg) \|\Lambda^{-s}\lf(\phi^2-1\rg)\|.
\end{aligned}$$
For $I_{17}$, noticing that
$$\int_{\mathbb{T}^3} \mathbf{u}'\cdot \Lambda^{-s}\nabla\lf(\phi^2-1\rg) \Lambda^{-s}\lf(\phi^2-1\rg) d\mathbf{x}=0, $$
where
$\displaystyle\mathbf{u}'=\bbint_{\mathbb{T}^3} \mathbf{u} d\mathbf{x}$, we have
$$\begin{aligned}
I_{18}&=\int_{\mathbb{T}^3}  \Lambda^{-s}\lf[(\mathbf{u}-\mathbf{u}')\cdot\nabla\lf(\phi^2-1\rg)\rg] \Lambda^{-s}\lf(\phi^2-1\rg) d\mathbf{x}\\
&\lesssim \|\Lambda^{-s}\lf[(\mathbf{u}-\mathbf{u}')\cdot \nabla\lf(\phi^2-1\rg)\rg]\| \|\Lambda^{-s}\lf(\phi^2-1\rg) \|\\
&
\overset{\eqref{hg25}}\lesssim\|\mathbf{u}-\mathbf{u}'\|_{L^{\f{3}{s}}}\|\nabla\lf(\phi^2-1\rg)\| \|\Lambda^{-s}\lf(\phi^2-1\rg) \|\\
&\overset{\eqref{hh45}, \eqref{h45}}
\lesssim \lf(\|\nabla\mathbf{u}\|_{H^2}^2+\|\nabla\phi\|^2\rg) \|\Lambda^{-s}\lf(\phi^2-1\rg)\|.
\end{aligned}$$
Last,  noticing that
$$\phi\Delta\phi=\f{1}{2}\Delta\lf(\phi^2-1\rg)+|\nabla\phi|^2, $$
we rewrite $I_{19}$ as
\begin{equation}\label{hr412}\begin{aligned}
I_{19}&=\underline{-\f{\lf(\overline{\phi^{2}}-1\rg)}{2\bar{\rho}}\int_{\mathbb{T}^3} |\Lambda^{-s}\nabla\lf(\phi^2-1\rg)|^2 d\mathbf{x}}_{I^1_{19}}\\&+\underline{\int_{\mathbb{T}^3} \Lambda^{-s}\lf[
\lf(\f{\lf(\phi^2-1\rg)}{2\rho}-\f{\lf(\overline{\phi^{2}}-1\rg)}{2\bar{\rho}}\rg)\Delta\lf(\phi^2-1\rg)\rg]\Lambda^{-s}\lf(\phi^2-1\rg)}_{I^2_{19}}\\
&
+\underline{\int_{\mathbb{T}^3} \Lambda^{-s}\lf[\f{\lf(\phi^2-1\rg)}{\rho}|\nabla\phi|^2\rg]\Lambda^{-s}\lf(\phi^2-1\rg) d\mathbf{x}}_{I^3_{19}},
\end{aligned}\end{equation} where $\displaystyle\overline{\phi^{2}}=\bbint_{\mathbb{T}^3} \phi^2 d\mathbf{x}.$  Then we have
$$I^1_{19} \overset{\eqref{hr411}}\lesssim \eta_1 \|\Lambda^{-s} \nabla\lf(\phi^2-1\rg)\|^2,$$
$$\begin{aligned}
I^2_{19}&\lesssim {\Big\|}\Lambda^{-s}\lf[
\lf(\f{\lf(\phi^2-1\rg)}{2\rho}-\f{\lf(\overline{\phi^{2}}-1\rg)}{2\bar{\rho}}\rg)\Delta\lf(\phi^2-1\rg)\rg]{\Big\|} \|\Lambda^{-s}\lf(\phi^2-1\rg) \|\\
&
\overset{\eqref{hg25}}\lesssim\lf(\|\rho-\bar{\rho}\|_{L^{\f{3}{s}}}+\|\phi^2-\overline{\phi^{2}}\|_{L^{\f{3}{s}}}\rg)\|\Delta\lf(\phi^2-1\rg)\| \|\Lambda^{-s}\lf(\phi^2-1\rg) \|\\
&\overset{\eqref{hh45}, \eqref{h45}}
\lesssim \lf(\|\nabla\sigma\|_{H^2}^2+\|\nabla\phi\|_{H^2}^2\rg) \|\Lambda^{-s}\lf(\phi^2-1\rg)\|,
\end{aligned}$$  and
$$\begin{aligned}
I^3_{19}&\lesssim {\Big\|}\Lambda^{-s}\lf[\f{\lf(\phi^2-1\rg)}{\rho}|\nabla\phi|^2\rg]{\Big\|} \|\Lambda^{-s}\lf(\phi^2-1\rg) \|\\
&
\overset{\eqref{hg25}, \eqref{hr33}, \eqref{hg313}}\lesssim\|\nabla\phi\|_{L^{\f{3}{s}}}\|\nabla\phi\| \|\Lambda^{-s}\lf(\phi^2-1\rg) \|\\
&\overset{\eqref{hh45}}
\lesssim \|\nabla\phi\|_{H^2}^2 \|\Lambda^{-s}\lf(\phi^2-1\rg)\|.
\end{aligned}$$
Therefore, we obtain from  \eqref{hr412} that
$$I_{19} \lesssim \eta_1 \|\Lambda^{-s} \nabla\lf(\phi^2-1\rg)\|^2+\lf(\|\nabla\sigma\|_{H^2}^2+\|\nabla\phi\|_{H^2}^2\rg) \|\Lambda^{-s}\lf(\phi^2-1\rg)\|.$$
Substituting the estimates on $I_{i} (i=15, \cdots, 19)$ into \eqref{hh44} yields \eqref{hr49}.
The proof of Lemma \ref{lem42} is completed.\end{proof}

\section{The proof of Theorem \ref{theo 2.1}}
\setcounter{equation}{0}

In this section, we shall combine all the energy estimates that we have derived in the previous
section and the Sobolev interpolation to prove Theorem \ref{theo 2.1}.

\subsection{The existence of global solutions}
\begin{proposition}  \label{estimate} \textbf{(a priori estimate).}
There exists a positive constant and $M_1>0$, if  $(\rho,	\mathbf{u}, \phi)\in X_{M_1}\big([0,T^*]\big)$, then
\begin{equation}\label{a priori estimate}
\begin{aligned}
&\sup_{t\in[0,T]}\Big\{\|(\rho-\bar{\rho}, \mathbf{u})(t)\|^2_{H^3}+ \|\nabla \phi(t)\|^2_{H^2}+ \|\phi^2(t)-1\|^2\Big\}\\
&\quad+\int_0^{T}\|\nabla\rho\|^2_{H^2} d\tau+ \int_0^{T}\|\nabla\mathbf{u}\|_{H^3}d\tau+\int_0^{T}\|\nabla\phi\|^2_{H^4} d\tau\\
&\displaystyle\quad\lesssim \|\rho_0-\bar{\rho}\|^2_{H^3}+\|\mathbf{u}_0\|^2_{H^3}+ \|\nabla \phi_0\|^2_{H^2}+\|\phi^2_0-1\|^2.
\end{aligned}
\end{equation}
\end{proposition}
\begin{proof}
We first close the energy estimates at each $l-$th level in our weak sense to prove \eqref{h18}.
Let $0\leq l\leq m$ with $1\leq m\leq 3$. Summing up the estimate \eqref{hr329} of Lemma \ref{lem33} for from $k=l$ to $m-1$, we obtain
	\begin{eqnarray}\label{h51}
 && \f{d}{dt}\sum_{l\leq k\leq m-1}\|\nabla^{k+1} \phi\|^2+\f{1}{4\bar{\rho}^2}\sum_{l+1\leq k\leq m}\|\nabla^{k+3}\phi\|^2 \notag\\
 &&\lesssim \eta_1\sum_{l+1\leq k\leq m}\lf(\|\nabla^{k}\sigma\|^2+\|\nabla^{k+1} \mathbf{u}\|^2\rg).
 \end{eqnarray}
Also, summing up the estimate \eqref{hr316} of Lemma \ref{lem32} for from $k=l$ to $m$, we obtain
\begin{equation}
\label{h52}
\begin{aligned}
\f{d}{dt}&\sum_{l\leq k\leq m}\lf(\|\nabla^k\mathbf{u}\|^2
+\f{p'(\bar{\rho})}{\bar{\rho}^2}\|\nabla^k\sigma\|^2\rg)+\f{\nu_0}{\bar{\rho}}\sum_{l\leq k\leq m}\|\nabla^{k+1} \mathbf{u}\|^2\\
&\lesssim \eta_1\sum_{l\leq k\leq m}\lf(\|\nabla^{k}
\sigma\|^2+\|\nabla^{k+2}\phi\|^2\rg).
\end{aligned}\end{equation}
Last, summing up the estimate \eqref{hr326} of Lemma \ref{lem35} for from $k=l$ to $m-1$, we obtain
 \begin{equation}
\label{h53}
\begin{aligned}
\f{d}{dt}&\sum_{l\leq k\leq m-1}\int_{\mathbb{T}^3} \nabla^{k} \mathbf{u}\cdot\nabla^{k+1}\sigma d\mathbf{x}+\f{p'(\bar{\rho})}{2\bar{\rho}}\sum_{l+1\leq k\leq m} \|\nabla^{k} \sigma\|^2
\\
&\lesssim \eta_1 \sum_{l+1\leq k\leq m}\lf(	\|\nabla^{k+1} \mathbf{u}\|^2+\|\nabla^{k+3}\phi\|^2\rg)+\sum_{l\leq k\leq m-1}\|\nabla^{k+1} \mathbf{u}\|^2.
\end{aligned}\end{equation}
Let $\epsilon\in (0, 1]$ be suitably small. Then, summing \eqref{h38}, \eqref{h51},  \eqref{h52} and $\epsilon\times \eqref{h53}$, and choosing $\eta_1>0$ to be small, we obtain
 \begin{equation}
\label{h54}
\f{d}{dt}\mathcal{E}_l^m(t)+\f{1}{2}\Lambda_{l}^m(t)\leq 0
\end{equation}for any $0\leq l< m\leq 3$, where
\begin{equation}\label{h55}\begin{aligned}&\begin{aligned}
\mathcal{E}_l^m(t):&=\sum_{l\leq k\leq m}\lf(\|\nabla^k\mathbf{u}\|^2
+\f{p'(\bar{\rho})}{\bar{\rho}^2}\|\nabla^k\sigma\|^2\rg)+\epsilon \sum_{l\leq k\leq m-1}\int_{\mathbb{T}^3} \nabla^{k} \mathbf{u}\cdot\nabla^{k+1}\sigma d\mathbf{x}
\\
&+\sum_{l\leq k\leq m-1}\|\nabla^{k+1} \phi\|^2+\int_{\mathbb{T}^3}\lf(\f{\rho}{2}\mathbf{u}^2+G(\rho)+\frac{1}{2}|\nabla
\phi|^2+\frac{\rho}{4}(\phi^2-1)^2\rg)d\mathbf{x},\end{aligned}
\\
&\begin{aligned}
\Lambda_{l}^m(t):&=\epsilon \f{p'(\bar{\rho})}{2\bar{\rho}}\sum_{l+1\leq k\leq m} \|\nabla^{k} \sigma\|^2+\f{1}{4\bar{\rho}^2}\sum_{l+1\leq k\leq m}\|\nabla^{k+2}\phi\|^2\\
&+\f{\nu_0}{\bar{\rho}}\sum_{l\leq k\leq m}\|\nabla^{k+1} \mathbf{u}\|^2-\epsilon C_0\sum_{l\leq k\leq m-1}\|\nabla^{k+1} \mathbf{u}\|^2 .\end{aligned}\end{aligned}
\end{equation}
Notice that since $\epsilon\in (0, 1]$ be suitably small, we obtain from \eqref{h55} that
\begin{equation}\label{h56}\begin{aligned}
&\mathcal{E}_l^m(t) \backsimeq \|\nabla^l(\sigma, \mathbf{u})(t)\|^2_{H^{m-l}}+\|\nabla^{l+1} \phi(t)\|^2_{H^{m-1-l}}+\|\phi^2-1\|^2,\\
&\Lambda_{l}^m(t)\backsimeq \|\nabla^l \sigma(t)\|^2_{H^{m-l}}+ \|\nabla^{l+1}\mathbf{u}(t)\|^2_{H^{m-l}}+ \|\nabla^{l+1} \phi(t)\|^2_{H^{m+1-l}}
\end{aligned} \end{equation} uniformly for all $t\geq 0$.
Now taking $l=0$ and $m=3$ in \eqref{h54}, and then  using \eqref{h56} and \eqref{h54}, we get
\begin{equation}\label{h57}
\begin{aligned}
&\|\sigma(t)\|^2_{H^{3}}+\|\mathbf{u}(t)\|^2_{H^{3}}+ \|\nabla \phi(t)\|^2_{H^{3-1}}+\|\phi^2-1\|^2\lesssim \mathcal{E}_0^3(t)\\&\lesssim \mathcal{E}_0^3(0)\lesssim \|\rho_0-\bar{\rho}\|^2_{H^{3}}+\|\mathbf{u}_0\|^2_{H^{3}}+ \|\nabla \phi_0\|^2_{H^{3-1}}+\|\phi_0^2-1\|^2.
\end{aligned}
\end{equation}
Using \eqref{g17} and \eqref{h57}, by a standard continuity argument, we can close the a priori estimate
\eqref{a priori estimate}.  This in turn allows us to take $l=0$ and $m=3$ in \eqref{h54}, and then integrate it directly in time to obtain
 \eqref{a priori estimate}.\end{proof}

\vspace{3ex} By using the Proposition 2.1, a solution for $t\in[0,T^*]$ can be obtained which satisfies \eqref{hg313} in $\mathbb{T}^3\times[0,T^*]$, $T^*$ only depends on the initial data  of the periodic boundary value problem \eqref{h21}. From and Proposition 3.1,  one can start again from $T^*$, by the same way, one can find a solution in $[0,2T^*]$, and so on. Thus the existence and uniqueness of the global solution for system \eqref{hg313} is obtained.

\subsection{Algebraic decay}
We   prove \eqref{h124}-\eqref{h125}. Define $$\mathcal{E}_{-s}:=\f{p'(\bar{\rho})}{\bar{\rho}^2}\|\Lambda^{-s} \sigma\|^2+\|\Lambda^{-s} \mathbf{u}\|^2+\|\Lambda^{-s} \nabla\phi\|^2+\|\Lambda^{-s}\lf(\phi^2-1\rg)\|^2 $$ for $s\in (0, \f{3}{2})$.
Then, integrating in time \eqref{hr41} and \eqref{hr49},  and using  \eqref{h18}, we get
$$\begin{aligned}
\mathcal{E}_{-s}(t)&\lesssim \mathcal{E}_{-s}(0)+ \int_0^t \lf(\|\nabla(\sigma,  \mathbf{u})\|_{H^2}^2+\|\nabla\phi\|_{H^3}^2\rg)\sqrt{\mathcal{E}_{-s}(\tau)}d\tau\\
&\lesssim 1+ \sup_{0\leq \tau\leq t}\sqrt{\mathcal{E}_{-s}(\tau)},
\end{aligned}$$ which implies  \eqref{h124}.
If $l=0,\cdots, 2$, we may use \eqref{hg24} to have
\begin{equation}\label{h512}
\|\nabla^{l+1}f\|\gtrsim\|\Lambda^{-s}f\|^{-\f{1}{l+s}}\|\nabla^{l}f\|^{1+\f{1}{l+s}}.\end{equation}
By \eqref{h512} and \eqref{h124}, we get
\begin{equation}\label{hh512}
\|\nabla\phi\|\overset{\eqref{hr33}}\gtrsim\|\nabla(\phi^2-1)\|\gtrsim\|\phi^2-1\|^{1+\f{1}{s}},\end{equation}
and
$$\|\nabla^{l+1}\sigma\|^2+\|\nabla^{l+1}\mathbf{u}\|^2+\|\nabla^{l+1}\nabla\phi\|^2\gtrsim \lf(\|\nabla^{l}\sigma\|^2+\|\nabla^{l}\mathbf{u}\|^2+\|\nabla^{l}\nabla\phi\|^2\rg)^{1+\f{1}{l+s}}, $$
which implies
\begin{equation}\label{h513}
\begin{aligned}
&\|\nabla ^{l+1} \sigma(t)\|^2_{H^{m-l-1}}+ \|\nabla^{l+1}\mathbf{u}(t)\|^2_{H^{m-l}}+ \|\nabla^{l+2} \phi(t)\|^2_{H^{m-l-1}}\\&
\gtrsim \lf(\|\nabla ^{l} \sigma(t)\|^2_{H^{m-l-1}}+ \|\nabla^{l}\mathbf{u}(t)\|^2_{H^{m-l}}+ \|\nabla^{l+1} \phi(t)\|^2_{H^{m-l-1}}\rg)^{1+\f{1}{l+s}}.
\end{aligned}
\end{equation} Also, using $\|\sigma(t)\|_{H^3}\lesssim 1$ due to \eqref{h18}, we have
\begin{equation}\label{hh513}
\|\nabla^m\sigma\|\gtrsim\|\nabla^m\sigma\|^{1+\f{1}{l+s}}.\end{equation}
By \eqref{hh512}-\eqref{hh513} and \eqref{hr317}, we  obtain from \eqref{h56} that
\begin{equation}\label{hg513}
\Lambda_{0}^1(t)\gtrsim\lf(\mathcal{E}_0^1(t)\rg)^{1+\f{1}{s}}.\end{equation}
Using  \eqref{hg513},  we deduce from \eqref{h54} with $l=0$
and $m=1$ that
$$\f{d}{dt}\mathcal{E}_0^1(t)+C_0\lf(\mathcal{E}_0^1(t)\rg)^{1+\f{1}{s}}\leq 0. $$
Solving this inequality directly gives (see \eqref{h119})
$$\mathcal{E}_0^1(t)\lesssim (1+t)^{-s}, $$
which implies  \eqref{hg124},  that is,
$$
\|(\sigma,\mathbf{u})(t)\|_{H^{1}}^2+\|\nabla \phi(t)\|^2 +\| \phi^2(t)-1\|^2\lesssim (1+t)^{-s},
$$due to $\eqref{h56}_1$.
We prove  \eqref{h125}.
For the small $\epsilon\in (0, 1]$,   summing   \eqref{h51}, \eqref{h52} and $\epsilon\times \eqref{h53}$, and choosing $\eta_1>0$ to be small, we have
\begin{equation}
\label{hg54}
\f{d}{dt}\mathcal{F}_l^m(t)+G_{l}^m(t)\leq 0,\ \  \mathrm{for\ any} \ 0\leq l< m\leq 3,
\end{equation} where
\begin{equation}\label{hg55}\begin{aligned}&\begin{aligned}
\mathcal{F}_l^m(t):&=\sum_{l\leq k\leq m}\lf(\|\nabla^k\mathbf{u}\|^2
+\f{p'(\bar{\rho})}{\bar{\rho}^2}\|\nabla^k\sigma\|^2\rg)
\\
&+\epsilon \sum_{l\leq k\leq m-1}\int_{\mathbb{T}^3} \nabla^{k} \mathbf{u}\cdot\nabla^{k+1}\sigma d\mathbf{x}+\sum_{l\leq k\leq m-1}\|\nabla^{k+1} \phi\|^2,\end{aligned}
\\
&\begin{aligned}
G_{l}^m(t):&=\epsilon \f{p'(\bar{\rho})}{2\bar{\rho}}\sum_{l+1\leq k\leq m} \|\nabla^{k} \sigma\|^2+\f{1}{4\bar{\rho}^2}\sum_{l+1\leq k\leq m}\|\nabla^{k+2}\phi\|^2\\
&+\lf(\f{\nu_0}{\bar{\rho}}-\epsilon C_0\rg)\sum_{l\leq k\leq m}\|\nabla^{k+1} \mathbf{u}\|^2.\end{aligned}\end{aligned}
\end{equation}
Notice that since $\epsilon\in (0, 1]$ be suitably small, we obtain from \eqref{hg55} that
\begin{equation}\label{hg56}\begin{aligned}
&\mathcal{F}_l^m(t) \backsimeq \|\nabla^l(\sigma, \mathbf{u})(t)\|^2_{H^{m-l}}+\|\nabla^{l+1} \phi(t)\|^2_{H^{m-1-l}},\\
&G_{l}^m(t)\backsimeq  \|\nabla ^{l} \sigma(t)\|^2_{H^{m-l}}+ \|\nabla^{l+1}\mathbf{u}(t)\|^2_{H^{m-l}}+ \|\nabla^{l+1} \phi(t)\|^2_{H^{m+1-l}},
\end{aligned} \end{equation} uniformly for all $t\geq 0$.
By \eqref{h513} and \eqref{hh513}, we  obtain from \eqref{hg56} that
\begin{equation}\label{hhg513}
G_{l}^m(t)\gtrsim\lf(\mathcal{F}_l^m(t)\rg)^{1+\f{1}{l+s}}.\end{equation}
Using  \eqref{hhg513},  we deduce from \eqref{hg54} with  $l=2$  and $m=3$ that
$$\f{d}{dt}\mathcal{F}_{2}^3(t)+C_0\lf(\mathcal{F}_{2}^3(t)\rg)^{1+\f{1}{3-1+s}}\leq 0. $$
Solving this inequality directly gives (see \eqref{h119})
$$\mathcal{F}_{2}^3(t)\lesssim (1+t)^{-(2+s)}, $$
which implies
\begin{equation}\label{hh520}
\|\nabla^{2}(\sigma,\mathbf{u})(t)\|_{H^{1}}^2+\|\nabla^{3} \phi(t)\|^2 \lesssim (2+t)^{-(l+s)},
\end{equation} due to $\eqref{hg56}_1$. Using \eqref{hr317}, we obtain \eqref{h125} from \eqref{hh520},
and so far, combining with Subsection 3.1,  Theorem 1.1 is obtained.

\section*{Conflict of Interests}
The authors declare that there is no conflict of interest regarding the publication of this paper.

\end{document}